\documentclass[onecolumn,11pt]{IEEEtran}

\usepackage{macro}

\title{Multiple Packing:\\Lower Bounds via Infinite Constellations}
\author{
\IEEEauthorblockN{
Yihan Zhang\IEEEauthorrefmark{1}
Shashank Vatedka\IEEEauthorrefmark{2}
}\\
\IEEEauthorblockA{
\IEEEauthorrefmark{1}Institute of Science and Technology Austria \\
\IEEEauthorrefmark{2}Department of Electrical Engineering, Indian Institute of Technology Hyderabad
}
}

\begin{document}
\maketitle

\begin{abstract}
We study the problem of high-dimensional multiple packing in Euclidean space. Multiple packing is a natural generalization of sphere packing and is defined as follows. Let $ N>0 $ and $ L\in\mathbb{Z}_{\ge2} $. A multiple packing is a set $\mathcal{C}$ of points in $ \mathbb{R}^n $ such that any point in $ \mathbb{R}^n $ lies in the intersection of at most $ L-1 $ balls of radius $ \sqrt{nN} $ around points in $ \mathcal{C} $. 
% We study the multiple packing problem for both bounded point sets whose points have norm at most $\sqrt{nP}$ for some constant $P>0$ and unbounded point sets whose points are allowed to be anywhere in $ \mathbb{R}^n $. 
Given a well-known connection with coding theory, multiple packings can be viewed as the Euclidean analog of list-decodable codes, which are well-studied for finite fields. 
In this paper, we derive the best known lower bounds on the optimal density of list-decodable infinite constellations for constant $L$ under a stronger notion called average-radius multiple packing. 
To this end, we apply tools from high-dimensional geometry and large deviation theory. 
% Our result corrects a mistake in a previous paper by Blinovsky \cite{blinovsky-2005-random-packing}. 
% These bounds are obtained via a proxy known as error exponent. The latter quantity is the best exponent of the probability of list-decoding error when the code is corrupted by a Gaussian noise. We establish a curious inequality which relates the error exponent, a quantity of average-case nature, to the list-decoding radius, a quantity of worst-case nature. We derive various bounds on the error exponent in both bounded and unbounded settings which are of independent interest beyond multiple packing.
\end{abstract}

% For arXiv
% We study the problem of high-dimensional multiple packing in Euclidean space. Multiple packing is a natural generalization of sphere packing and is defined as follows. Let $ N>0 $ and $ L\in\mathbb{Z}_{\ge2} $. A multiple packing is a set $\mathcal{C}$ of points in $ \mathbb{R}^n $ such that any point in $ \mathbb{R}^n $ lies in the intersection of at most $ L-1 $ balls of radius $ \sqrt{nN} $ around points in $ \mathcal{C} $. Given a well-known connection with coding theory, multiple packings can be viewed as the Euclidean analog of list-decodable codes, which are well-studied for finite fields. In this paper, we derive the best known lower bounds on the optimal density of list-decodable infinite constellations for constant $L$ under a stronger notion called average-radius multiple packing. To this end, we apply tools from high-dimensional geometry and large deviation theory. 

% \newpage
% \tableofcontents
% \newpage

\section{Introduction}
\label{sec:intro}
We {study} the problem of \emph{multiple packing} {in Euclidean space}, a natural generalization of the sphere packing problem~\cite{conway-sloane-book}. 
Let $ N>0 $ and $ L\in\bZ_{\ge2} $. 
We say that a point set $\cC$ in $ \bR^n $ forms a \emph{$(N,L-1)$-multiple packing}\footnote{We choose to stick with $L-1$ rather than $L$ for notational convenience. 
This is because in the proof, we need to examine the violation of $(L-1)$-packing, i.e., the existence of an $L$-sized subset that lies in a ball of radius $\sqrt{nN}$. } if any point in $ \bR^n $ lies in the intersection of at most $L-1$ balls of radius $ \sqrt{nN} $ around points in $\cC$. 
Equivalently, the radius of the smallest ball containing any size-$L$ subset of $\cC$ is larger than $ \sqrt{nN} $. 
This radius is known as the \emph{Chebyshev radius} of the $L$-sized subset. 
If $ L = 2 $, then $\cC$ forms a \emph{sphere packing}, i.e., a point set such that balls of radius $ \sqrt{nN} $ around points in $\cC$ are disjoint, or equivalently, the pairwise distance of points in $\cC$ is larger than $ 2\sqrt{nN} $. 
The density of $\cC$ is measured by  \emph{rate} (a.k.a.\ the \emph{normalized logarithmic density (NLD)}) defined as\footnote{Logarithms to the base $e$ are denoted by $\ln(\cdot)$. }%\footnote{Here we use $ \cB^n(r) $ to denote an $n$-dimensional Euclidean ball of radius $r$ centered at the origin.}  
\begin{align}
R(\cC) \coloneqq \limsup_{K\to\infty}\frac{1}{n}\ln\frac{\card{\cC\cap[-K,K]^n}}{\card{[-K,K]^n}}, 
\label{eqn:intro-rate-unbdd}
\end{align}
i.e., the (normalized) number of points per volume. 
Denote by $ C_{L-1}(N) $ the largest rate of a $(N,L-1)$-multiple packing as $ n\to\infty $. 
% Note that $ C_{L-1}(N) $ depends on $P$ and $N$ only through their ratio $ N/P $ which we call the \emph{noise-to-signal ratio}.
The goal of this paper is to advance the understanding of $ C_{L-1}(N) $.

The problem of multiple packing is closely related to the \emph{list-decoding} problem \cite{elias-1957-listdec,wozencraft-1958-listdec} in coding theory. 
Indeed, a multiple packing can be seen exactly as the \emph{Euclidean} analog of a list-decodable code.
% Specifically, 
% consider a channel that takes as input a codeword of Euclidean norm at most $ \sqrt{nP} $ and is governed by an adversary who can add a noise vector of length less than $ \sqrt{nN} $ to the transmitted codeword. 
% A multiple packing can be used as a code over this a channel so that it attains zero error under $ (L-1) $-list-decoding. 
% That is, any output of the channel can always be decoded to a list of at most $ L-1 $ codewords, including the transmitted one, that are less than $ \sqrt{nN} $ away from the channel output. 
We will interchangeably use the terms ``packing'' and ``code'' to refer to the point set of interest. 
To see the connection, note that if any point in a multiple packing is transmitted through an adversarial channel that can inflict an arbitrary additive noise of length at most $ \sqrt{nN} $, then given the distorted transmission, one can decode to a list of the nearest $L-1$ points which is guaranteed to contain the transmitted one. 
The quantity $ C_{L-1}(N) $ can therefore be interpreted as the capacity of this channel (in the sense of Poltyrev \cite{poltyrev1994coding}). 
Moreover, list-decodable codes can be turned into unique-decodable codes with the aid of side information such as common randomness shared between the transmitter and receiver \cite{langberg-focs2004,sarwate-thesis,bhattacharya2019shared}. 
List-decoding also serves as a proof technique towards unique-decoding in various communication scenarios; see, e.g., \cite{zhang-quadratic-arxiv,zhang-2020-twoway}.

For $L=2$, the sphere packing problem 
% was extensively studied in the mathematics community. 
% The $ L=2 $ case 
has a long history since at least the Kepler conjecture \cite{kepler-1611} in 1611. 
The best known lower bound is due to Minkowski \cite{minkowski-sphere-pack} using a straightforward volume packing argument. 
The best known upper bound is obtained by reducing it to the bounded case (i.e., packing points in a ball rather than in $ \bR^n $) for which we have the Kabatiansky--Levenshtein linear programming-type bound \cite{kabatiansky-1978}. 
For $ L>2 $, Blinovsky \cite{blinovsky-2005-random-packing} claimed a lower bound by analyzing an (expurgated) Poisson Point Process (PPP). 
However, we noticed some gaps in the proof (see \Cref{sec:bli-mistake-ppp}). 
In this work we use a different approach to construct an unbounded packing which achieves the same lower bound as claimed in \cite{blinovsky-2005-random-packing}. 
The paper \cite{blinovsky-2005-random-packing} also presented an Elias--Bassalygo-type bound without a proof. 
% We present a proof for it in \Cref{sec:eb} for completeness. 
A complete proof of it can be found in \cite{zhang-split-misc}. 

For the multiple packing problem with $ L>2 $, many existing lower bounds 
% (for both bounded and unbounded settings) 
are obtained under a stronger notion known as the \emph{average-radius} multiple packing (see \Cref{def:avg-rad-multi-pack-unbounded} for the exact definition). 
A set $ \cC $ of $ \bR^n $-valued points is called an average-radius multiple packing if for any $(L-1)$-subset of $\cC$, the maximum distance from any point in the subset to the centroid of the subset is less than $ \sqrt{nN} $. 
Here the centroid of a subset is defined as the average of the points in the subset. 
Denote by $ \ol C_{L-1}(N) $ the largest density of average-radius multiple packings.
In fact, we study this stronger notion of multiple packing in the present paper.
% In this paper, we treat it as a different notion from $ C_{L-1}(N) $, and we study both notions. 
For any finite $ L\in\bZ_{\ge2} $, it is unknown whether the largest multiple packing density under the regular notion is the same as that under the average-radius variant. 

For $ L\to\infty $, Zhang and Vatedka \cite{zhang-vatedka-2019-listdecreal} determined the limiting value of $ C_{L-1}(N) $. 
It follows from results in this paper that $ \ol C_{L-1}(N) $ converges to the same value as $ L\to\infty $. 

Very little is known about structured packings. 
Grigorescu and Peikert \cite{grigorescu-peikert-2012-list-dec-barnes-wall} initiated the study of list-decodability of lattices. 
See also the recent work \cite{mook-peikert-2020-lattice} by Mook and Peikert. 
Zhang and Vatedka \cite{zhang-vatedka-2019-listdecreal} had results on list-decodability of random lattices. 

\subsection*{Relation to conference version}
This work was presented in part at the 2022 IEEE International Symposium on Information Theory~\cite{zhang-ppp-isit}. 
\cite{zhang-ppp-isit} only contains the proof of \Cref{eqn:compare-lb-ppp} using PPPs. 
In the current paper, the same result is obtained via infinite constellations whose analysis is simpler and more transparent. 
Furthermore, results on fundamental properties of different notions of packing density and radius are presented.

\section{Related works}
\label{sec:related}
For $L=2$, the problem of sphere packing has a long history and has been extensively studied, especially for small dimensions. 
The largest packing density is open for almost every dimension, except for $ n = 1 $ (trivial), $2$ (\cite{thue1911-2dspherepacking,toth-1940-2dspherepacking}), $ 3 $ (the Kepler conjecture, \cite{hales1998kepler,hales2017formal}), $ 8 $ (\cite{viazovska-2017-8dspherepacking}) and $24$ (\cite{cohn-2017-24spherepacking}). 
For $ n\to\infty $, the best lower and upper bounds remain the trivial sphere packing bound \cite{minkowski-sphere-pack} and Kabatiansky--Levenshtein's linear programming bound \cite{kabatiansky-1978}. 
This paper is only concerned with (multiple) packings in high dimensions and we measure the density in the normalized way as mentioned in \Cref{sec:intro}. 

There is a parallel line of research in combinatorial coding theory. 
Specifically, a uniquely-decodable code (resp. list-decodable code) is nothing but a sphere packing (resp. multiple packing) which has been extensively studied for $ \bF_q^n $ equipped with the Hamming metric. 
Empirically, it seems that the problem is harder for smaller field sizes $q$.

We first list the best known results for sphere packing (i.e., $L=2$) in Hamming spaces. 
For $q=2$, the best lower and upper bounds are the Gilbert--Varshamov bound \cite{gilbert1952,varshamov1957} proved using a trivial volume packing argument and the second MRRW bound \cite{mrrw2} proved using the seminal Delsarte's linear programming framework \cite{delsarte-1973}, respectively. 
Surprisingly, the Gilbert--Varshamov bound can be improved using algebraic geometry codes \cite{goppa1977,tvz} for $ q\ge49 $. 
Note that such a phenomenon is absent in $ \bR^n $; as far as we know, no algebraic constructions of Euclidean sphere packings are known to beat the greedy/random constructions. 
For $ q\ge n $, the largest packing density is known to exactly equal the Singleton bound \cite{komamiya1953singleton,joshi1958singleton,singleton1964} which is met by, for instance, the Reed--Solomon code \cite{reed-solomon}. 

Less is known for multiple packing in Hamming spaces. 
We first discuss the binary case (i.e., $q=2$). 
For every $ L\in\bZ_{\ge2} $, the best lower bound appears to be Blinovsky's bound \cite[Theorem 2, Chapter 2]{blinovsky2012book} proved under the stronger notion of average-radius list-decoding. 
The best upper bound for $L=3$ is due to Ashikhmin, Barg and Litsyn \cite{abl-2000-list-size-2} who combined the MRRW bound \cite{mrrw2} and Litsyn's bound \cite{litsyn-1999} on distance distribution. 
For any $ L\ge4 $, the best upper bound is essentially due to Blinovsky again \cite{blinovsky-1986-ls-lb-binary}, \cite[Theorem 3, Chapter 2]{blinovsky2012book}, though there are some partial improvements. 
In particular, the idea in \cite{abl-2000-list-size-2} was recently generalized to larger $L$ by Polyanskiy \cite{polyanskiy-2016-list-dec} who improved Blinovsky's upper bound for \emph{even} $L$ (i.e., odd $L-1$) and sufficiently large $R$. 
Similar to \cite{abl-2000-list-size-2}, the proof also makes use of a bound on distance distribution due to Kalai and Linial \cite{kalai-linial-1995-distance-distribution} which in turn relies on Delsarte's linear programming bound. 
For larger $q$, Blinovsky's lower and upper bounds\footnotemark{} \cite{blinovsky-2005-ls-lb-qary,blinovsky-2008-ls-lb-qary-supplementary}, \cite[Chapter III, Lecture 9, \S 1 and 2]{ahlswede-blinovsky-2008-book} remain the best known. 
\footnotetext{Some gaps in the proof of the upper bound in \cite{blinovsky-2005-ls-lb-qary,blinovsky-2008-ls-lb-qary-supplementary} are recently observed. 
These gaps are closed in \cite{resch-yuan-zhang-plotkin} and the results therein are extended to the \emph{list-recovery} setting which is a generalization of $q$-ary list-decoding.}

As $L\to\infty$, the limiting value of the largest multiple packing density is a folklore in the literature known as the ``list-decoding capacity'' theorem\footnote{It is an abuse of terminology to use ``list-decoding capacity'' here to refer to the large $L$ limit of the $(L-1)$-list-decoding capacity.}.
Moreover, the limiting value remains the same under the average-radius notion. 

The problem of list-decoding was also studied for settings beyond the Hamming errors, e.g., list-decoding against erasures \cite{guruswami-it2003,ben-aroya-doron-ta-shma-2018-explicit-erasure-ld}, insertions/deletions \cite{guruswami-2020-listdec-insdel}, asymmetric errors \cite{polyanskii-zhang-2021-z}, etc. 
Zhang et al.\ considered list-decoding over general adversarial channels \cite{zhang-2019-list-dec-general}. 
List-decoding against other types of adversaries with \emph{limited} knowledge such as oblivious or myopic adversaries were also considered in the literature \cite{hughes-1997-list-avc,sarwate-gastpar-2012-listdec,zhang-2020-obli-list-dec,hosseinigoki-kosut-2018-oblivious-gaussian-avc-ld,zhang-quadratic-arxiv}. 
The current paper can be viewed as a collection of results for list-decodable codes for adversarial channels over $ \bR $ with $ \ell_2 $ constraints.

\section{Our results}
\label{sec:results}

We derive the best known lower bound on the largest multiple packing density.
% We compare various bounds that appear in this paper. 
% Let $ C_{L-1}(P,N) $ and $ \ol C_{L-1}(P,N) $ denote the largest density of multiple packings under the standard and the average-radius notions, respectively.
Let $ C_{L-1}(N) $ and $ \ol C_{L-1}(N) $ denote the largest density of multiple packings under the standard and the average-radius notions, respectively.

% \subsection{{Unbounded packings}}
% \label{sec:results-unbdd}
We juxtapose our bound with various existing bounds for the $(N,L-1)$-multiple packing problem. 
In \Cref{thm:lb-ppp}, we prove the following lower bound on the optimal density for $(N,L-1)$-\emph{average-radius} list-decoding (which is stronger than $ (N,L-1) $-list-decoding): 
\begin{align}
\ol C_{L-1}(N) &\ge \frac{1}{2}\ln\frac{L-1}{2\pi eNL} - \frac{\ln L}{2(L-1)}. \label{eqn:compare-lb-ppp}
\end{align}
This bound turns out to be the largest known lower bound on both $ C_{L-1}(N) $ and $ \ol{C}_{L-1}(N) $ for all $N\ge0$ and $ L\in\bZ_{\ge2} $.
In \cite{blinovsky-2005-random-packing}, Blinovsky considered PPPs and arrived at the same bound. 
See \Cref{sec:bli-mistake-ppp} for a discussion. 
Curiously, the above bound can also be obtained under $(N,L-1)$-list-decoding (which is weaker than $(N,L-1)$-average-radius list-decoding) via a connection with error exponents \cite{zhang-split-ee}. 
The techniques for \emph{bounded} packings (in which all points lie\footnotemark{} either in $ \cB^n(\vzero,\sqrt{nP}) $ or on $ \cS^{n-1}(\vzero,\sqrt{nP}) $ for some $P>0$) in \cite{blachman-few-1963-multiple-packing} can be adapted to the unbounded setting (where points can lie anywhere in $ \bR^n $) considered in this paper and be strengthened to work for the stronger notion of \emph{average-radius} multiple packing. 
\footnotetext{Here we use $ \cB^n(\vx,r) $ and $ \cS^{n-1}(\vx,r) $ to denote the $n$-dimensional Euclidean ball and $ (n-1) $-dimensional Euclidean sphere of radius $r$ centered at the $\vx$, respectively. }
They yield the following lower bound on $ \ol C_{L-1}(N) $: 
\begin{align}
\ol C_{L-1}(N) &\ge \frac{1}{2}\ln\frac{L-1}{4\pi eNL}. \label{eqn:compare-lb-bf-unbdd} 
\end{align}
As for upper bound, the techniques in \cite{blachman-few-1963-multiple-packing,blinovsky-1999-list-dec-real,blinovsky-2005-random-packing} can be adapted to the unbounded setting as well which yield the following upper bound on $ C_{L-1}(N) $: 
\begin{align}
C_{L-1}(N) &\le \frac{1}{2}\ln\frac{L-1}{2\pi eNL}. \label{eqn:compare-eb-unbdd}
\end{align}
Finally, it is known (see, e.g., \cite{zhang-vatedka-2019-listdecreal}) that as $ L\to\infty $, $ C_{L-1}(N) $ converges to the following expression: 
\begin{align}
C_{\mathrm{LD}}(N) &= \frac{1}{2}\ln\frac{1}{2\pi eN}. \label{eqn:compare-ld-cap-unbdd}
\end{align}
Note that, % the bound in \Cref{eqn:compare-lb-ppp} converges to \Cref{eqn:compare-ld-cap-unbdd} as $ L\to\infty $. 
by the lower and upper bounds (\Cref{eqn:compare-lb-ppp,eqn:compare-eb-unbdd}) on $ \ol C_{L-1}(N) $ for finite $L$, the limiting value of $ \ol C_{L-1}(N) $ as $ L\to\infty $ is also the above expression. 

All the above bounds for $ (N,L-1) $-multiple packing are plotted in \Cref{fig:ld-unbdd} with $ L = 5 $. 
The horizontal axis is $N$ and the vertical axis is the value of various bounds. 
The largest lower bound turns out to be \Cref{eqn:compare-lb-ppp} (for all $N\ge0$ and $ L\in\bZ_{\ge2} $).
This bound together with the Elias--Bassalygo-type upper bound in \Cref{eqn:compare-eb-unbdd} are plotted in \Cref{fig:PPPEB} for $ L=3,4,5 $. They both converge from below to \Cref{eqn:compare-ld-cap-unbdd} as $ L $ increases.

\begin{figure}[htbp]
	\centering
	\includegraphics[width=0.95\textwidth]{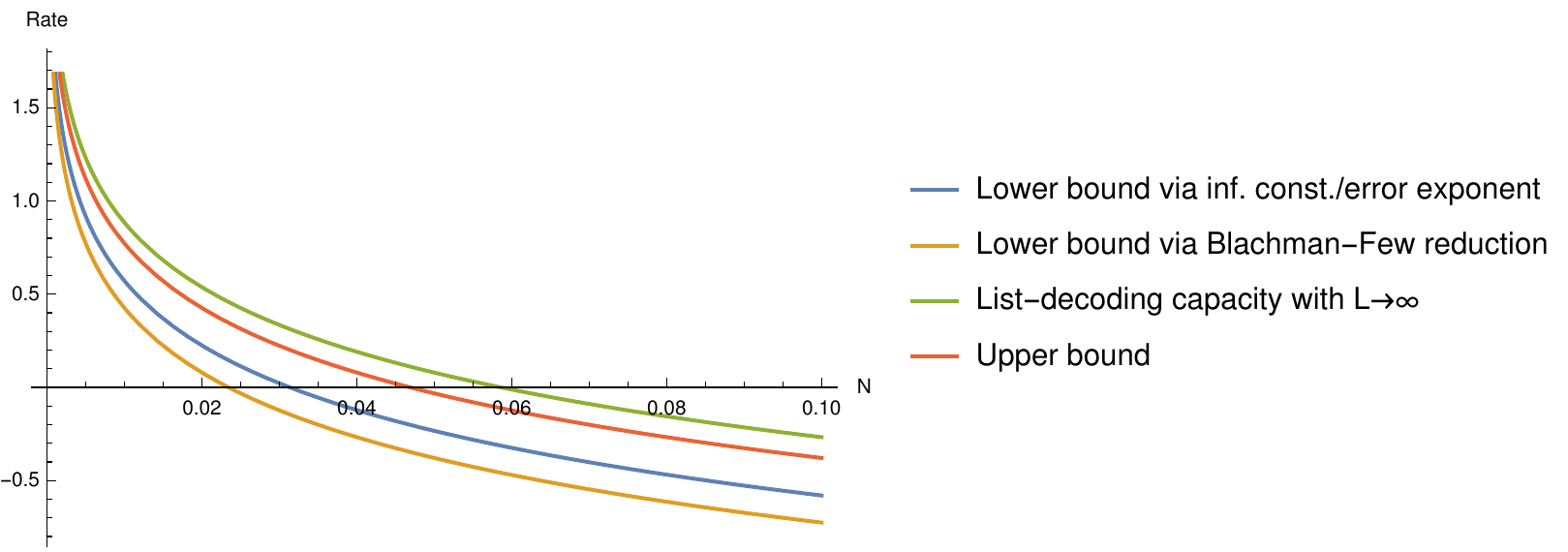}
	\caption{Comparison of different bounds for the $ (N,L-1) $-list-decoding problem. The horizontal axis is $N$ and the vertical axis is the value of bounds. We plot bounds for $ L = 5 $. Recall that the rate (\Cref{eqn:intro-rate-unbdd}) of a multiple packing is defined as the (normalized) number of points per volume which can be negative. The expressions for all bounds can be found in \Cref{eqn:compare-lb-ppp,eqn:compare-lb-bf-unbdd,eqn:compare-eb-unbdd,eqn:compare-ld-cap-unbdd}. }
	\label{fig:ld-unbdd}
\end{figure}

\begin{figure}[htbp]
	\centering
	\includegraphics[width=0.95\textwidth]{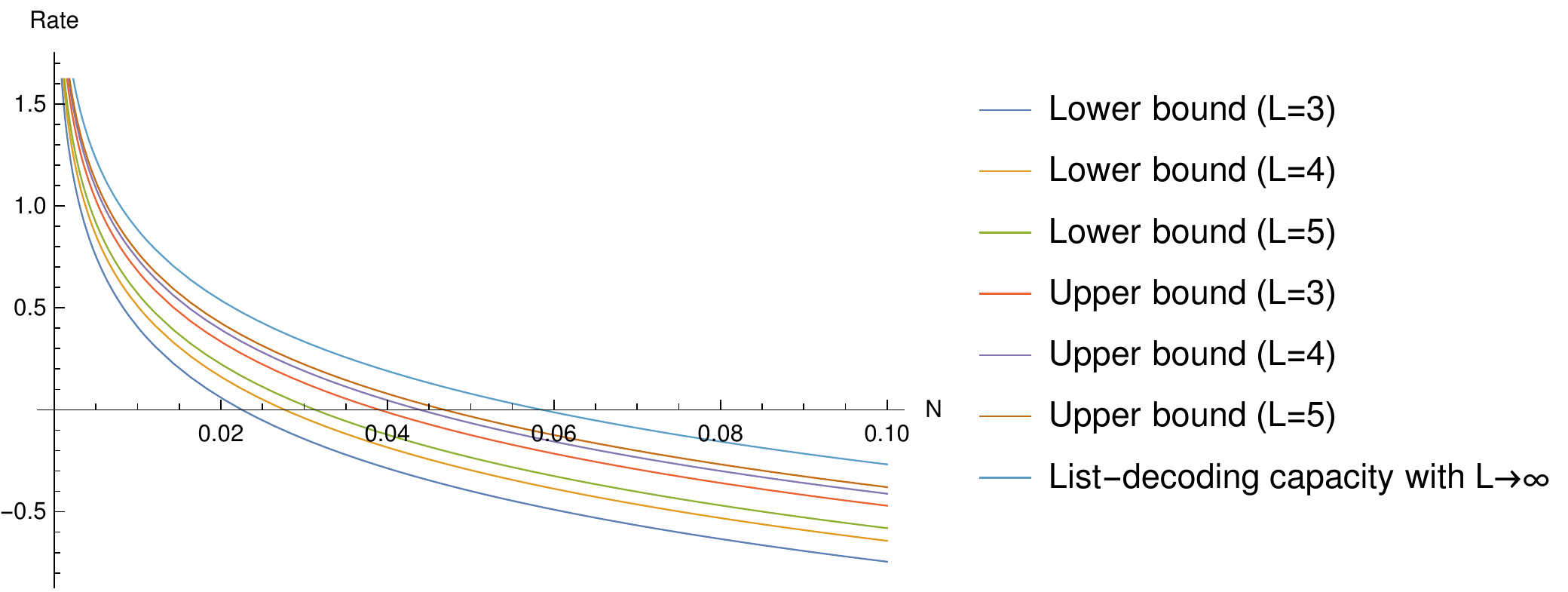}
	\caption{Plots of the best known lower bound (\Cref{eqn:compare-lb-ppp}) on $ C_{L-1}(N) $ and the Elias--Bassalygo-type upper bound for $ L=3,4,5 $. As $ L $ increases, they both converge from below to $ C_{\mathrm{LD}}(N) $ (\Cref{eqn:compare-ld-cap-unbdd}). 
	The lower bound \Cref{eqn:compare-lb-ppp} derived in this paper is under the \emph{average-radius} notion of multiple packing.
	Moreover, it can be obtained using a connection with error exponents \cite{zhang-split-ee} under the standard notion of multiple packing. }
	\label{fig:PPPEB}
\end{figure}

\section{List-decoding capacity for large $L$}
\label{sec:listdec-cap-large}

All bounds in this paper hold for any \emph{fixed} $L$. 
In this section, we discuss the impact of our finite-$L$ bounds on the understanding of the limiting values of the largest multiple packing density as $ L\to\infty $. 
Some of these results were known previously and others follow from the bounds in the current paper.

Characterizing % $ C_{L-1}(P,N) $, $ \ol C_{L-1}(P,N) $,
$ C_{L-1}(N) $ or $ \ol C_{L-1}(N) $ is a difficult task that is out of reach given the current techniques. 
However, if the list-size $L$ is allowed to grow, we can actually characterize 
\begin{align}
% C_{\mathrm{LD}}(P,N) &\coloneqq \lim_{L\to\infty} C_{L-1}(P,N), \quad
C_{\mathrm{LD}}(N) \coloneqq \lim_{L\to\infty} C_{L-1}(N), \quad
\ol C_{\mathrm{LD}}(N) \coloneqq \lim_{L\to\infty} \ol C_{L-1}(N), \notag 
\end{align}
where the subscript $ \mathrm{LD} $ denotes List-Decoding. 
% However, we are not able to characterize 
% \begin{align}
% \ol C_{\mathrm{LD}}(P,N) &\coloneqq \lim_{L\to\infty} \ol C_{L-1}(P,N). \notag 
% \end{align}
% We leave this as an intriguing open question (see \Cref{sec:open}). 

% It is well-known that $ C_{\mathrm{LD}}(P,N) = \frac{1}{2}\ln\frac{P}{N} $. 
% Specifically, the following theorem appears to be a folklore in the literature and a complete proof can be found in \cite{zhang-quadratic-arxiv}. 
% \begin{theorem}[Folklore, \cite{zhang-quadratic-arxiv}]
% \label{thm:listdec-cap-large-bdd}
% Let $ 0<N\le P $. 
% Then for any $ \eps>0 $, 
% \begin{enumerate}
% 	\item There exist $ (P,N,L-1) $-multiple packings of rate $ \frac{1}{2}\ln\frac{P}{N} - \eps $ for some $ L = \cO\paren{\frac{1}{\eps}\ln\frac{1}{\eps}} $; 
	
% 	\item Any $ (P,N,L-1) $-multiple packing of rate $ \frac{1}{2}\ln\frac{P}{N} + \eps $ must satisfy $ L = e^{\Omega(n\eps)} $. 
% \end{enumerate}
% Therefore, $ C_{\mathrm{LD}}(P,N) = \frac{1}{2}\ln\frac{P}{N} $. 
% \end{theorem}

The value of $ C_{\mathrm{LD}}(N) $ is characterized in \cite{zhang-vatedka-2019-listdecreal} which equals $ \frac{1}{2}\ln\frac{1}{2\pi eN} $. 
\begin{theorem}[\cite{zhang-vatedka-2019-listdecreal}]
\label{thm:listdec-cap-large-unbdd}
Let $ N>0 $. 
Then for any $ \eps>0 $, 
\begin{enumerate}
	\item There exist $ (N,L-1) $-multiple packings of rate $ \frac{1}{2}\ln\frac{1}{2\pi eN} - \eps $ for some $ L = \cO\paren{\frac{1}{\eps}\ln\frac{1}{\eps}} $; 
	
	\item Any $ (N,L-1) $-multiple packing of rate $ \frac{1}{2}\ln\frac{1}{2\pi eN} + \eps $ must satisfy $ L = e^{\Omega(n\eps)} $. 
\end{enumerate}
Therefore, $ C_{\mathrm{LD}}(N) = \frac{1}{2}\ln\frac{1}{2\pi eN} $. 
\end{theorem}

Moreover, we claim $ \ol C_{\mathrm{LD}}(N) = \frac{1}{2}\ln\frac{1}{2\pi eN} $. 
For an upper bound, recall that average-radius list-decodability implies (regular) list-decodability. 
Therefore, any upper bound on $ C_{L-1}(N) $ is also an upper bound on $ \ol C_{L-1}(N) $. 
We already saw an upper bound on $ C_{L-1}(N) $ in \Cref{eqn:compare-eb-unbdd} that approaches $ \frac{1}{2}\ln\frac{1}{2\pi eN} $ as $L\to\infty$. 
%On the other hand, we have seen constructions of $ (N,L-1) $-average-radius list-decodable codes in that attains rate approaching $ \frac{1}{2}\ln\frac{1}{2\pi eN} $ as $ L\to\infty $. 
Indeed, according to \Cref{thm:lb-ppp}, for sufficiently large $L$, our construction achieves $ \frac{1}{2}\ln\frac{1}{2\pi eN} $ under average-radius multiple packing. 

\begin{theorem}
\label{thm:avgrad-listdec-cap-large-unbdd}
For any $ N>0 $, $ \ol C_{\mathrm{LD}}(N) = \frac{1}{2}\ln\frac{1}{2\pi eN} $. 
\end{theorem}

\section{Our techniques}
\label{sec:techniques}

We summarize our techniques below. 

To obtain lower bounds on the largest multiple packing density, our basic strategy is random coding with expurgation, a standard tool from information theory. 
To show the existence of a list-decodable code of rate $R$, we simply randomly sample $ e^{nR} $ points independently each according to a certain distribution. 
There might be bad lists of size $L$ that violates the multiple packing condition. 
We then throw away (a.k.a.\ expurgate) one point from each of the bad lists. 
By carefully analyzing the error event and choosing a proper rate, we can guarantee that the remaining code has essentially the same rate after the removal process. 
We then get a list-decodable code of rate $R$ by noting that the remaining code contains no bad lists.

In the above framework, the key ingredient is a good estimate on the probability of the error event, i.e., the probability that the \emph{list-decoding radius} of a size-$L$ list is smaller than $ \sqrt{nN} $. 
Under the standard notion of multiple packing, the list-decoding radius is the \emph{Chebyshev radius} of the list, i.e., the radius of the smallest ball containing the list. 
Under the average-radius notion of multiple packing, the (squared) list-decoding radius is the \emph{average squared radius} of the list, i.e., the average squared distance from each point in the list to the centroid of the list. 

Using the above idea, we first construct a \emph{finite} codebook with minimum average squared radius $nN$ and supported over the hypercube $[-K,K]^n$ for a suitably chosen $K$. This is obtained by expurgating a random codebook obtained by choosing points independently and uniformly from $[-K,K]^n$. The finite codebook is then tiled across $\bR^n$ to obtain an infinite constellation with the aforementioned density and minimum average squared radius $nN$. This construction is loosely inspired by the infinite constellations~\cite{poltyrev1994coding} which was originally studied in the context of coding for the additive white Gaussian noise channel. A similar construction was used by~\cite{zhang-vatedka-2019-listdecreal} to derive lower bounds on $C_{\mathrm{LD}}(N)$ for large $L$.

%Due to the analytic simplicity of the average-radius variant, we manage to obtain the \emph{exact} rate of expurgated Poisson Point Processes (PPPs).
 % in the unbounded setting. 
The exact exponents of the probability of the error event are obtained using \cramer's large deviation principle and the Laplace's method. 

As a technical contribution, we discover several new representations of the average radius and the Chebyshev radius. 
They play crucial roles in facilitating the analyses and the applications of some of these representations go beyond the scope of this paper. 
To name a few, the  average squared radius of a list can be written as a quadratic form associated with the list. 
This representation is used to analyze 
Gaussian codes and spherical codes in \cite{zhang-split-misc} and infinite constellations in this paper. 
The  average squared radius can also be written as the difference between the average (squared) norm of points in the list and the (squared) norm of the centroid of the list. 
This representation is used to analyze spherical codes and ball codes in \cite{zhang-split-misc}. 
The average squared radius can be further written as the average pairwise distance of the list. 
This allows us to give a one-line proof of the Blachman--Few reduction and its strengthened version \cite{zhang-split-misc}. 
Yet another way of writing the average squared radius using the average norm and the average pairwise correlation turns out to be useful for the proof of the Plotkin-type bound \cite{zhang-split-misc}.

\section{Organization of the paper}
\label{sec:org}

This paper derives a lower bound on the largest multiple packing density which turns out to be the best known so far. 
The rest of the paper is organized as follows. 
Notational conventions and preliminary definitions/facts are listed in \Cref{sec:notation,sec:prelim}, respectively. 
After that, we present in \Cref{sec:def} the formal definitions of multiple packing and pertaining notions. 
We also discuss different notions of density of codes used in the literature. 
Furthermore, we obtain several novel representations of the Chebyshev radius and the average squared radius which are crucial for estimating their tail probabilities. 
% In \Cref{sec:rand-cod-exp}, we introduce the basic technology for proving lower bounds: random coding with expurgation. 
% The rest sections are essentially disjoint and can be read in an arbitrary order. 
% The rest of the sections are relatively independent of each other and can be read in an arbitrary order.
% In \Cref{sec:lb-gaussian}, we analyze the exact asymptotics of expurgated Gaussian codes under average-radius list-decoding. 
% In \Cref{sec:lb-spherical}, we obtain bounds on the rate of average-radius list-decodable expurgated spherical codes and ball codes. 
In \Cref{sec:lb-ppp}, we formally introduce our construction and prove the main result. 
% In \Cref{sec:lb-blachman-few}, we rederive, simplify and strengthen a lower bound by Blachman and Few \cite{blachman-few-1963-multiple-packing} using a novel representation of average squared radius in \Cref{sec:comments-radius}. 
% In \Cref{sec:lb-ee}, we prove the inequality that relates the Chebyshev radius to error exponent and combine it with bounds on error exponent to obtain lower bounds on the largest multiple packing density. 
% The bounds on error exponent used in this section are proved in \Cref{sec:ld-ee} for the bounded case and in \Cref{sec:ld-ee-unbdd} for the unbounded case. 
% Complete and clean proofs of the Plotkin-type bound and the Elias--Bassalygo-type bound are presented in \Cref{sec:ub}. 
We end the paper with several open questions in \Cref{sec:open}.

\section{Notation}
\label{sec:notation}
\noindent\textbf{Conventions.}
Sets are denoted by capital letters in calligraphic typeface, e.g., $\cC,\cB$, etc. 
Random variables are denoted by lower case letters in boldface or capital letters in plain typeface, e.g., $\bfx,S$, etc. Their realizations are denoted by corresponding lower case letters in plain typeface, e.g., $x,s$, etc. Vectors (random or fixed) of length $n$, where $n$ is the blocklength without further specification, are denoted by lower case letters with  underlines, e.g., $\vbfx,\vbfg,\vx,\vg$, etc. 
Vectors of length different from $n$ are denoted by an arrow on top and the length will be specified whenever used, e.g., $ \vec t, \vec\alpha $, etc. 
The $i$-th entry of a vector $\vx\in\cX^n$ is denoted by $\vx(i)$ since we can alternatively think of $\vx$ as a function from $[n]$ to $\cX$. Same for a random vector $\vbfx$. Matrices are denoted by capital letters, e.g., $A,\Sigma$, etc. Similarly, the $(i,j)$-th entry of a matrix $G\in\bF^{n\times m}$ is denoted by $G(i,j)$. We sometimes write $G_{n\times m}$ to explicitly specify its dimension. For square matrices, we write $G_n$ for short. Letter $I$ is reserved for identity matrix.  

\noindent\textbf{Functions.}
We use the standard Bachmann--Landau (Big-Oh) notation for asymptotics of real-valued functions in positive integers. 

For two real-valued functions $f(n),g(n)$ of positive integers, we say that $f(n)$ \emph{asymptotically equals} $g(n)$, denoted $f(n)\asymp g(n)$, if 
\[\lim_{n\to\infty}\frac{f(n)}{g(n)} = 1.\]
For instance, $2^{n+\log n}\asymp2^{n+\log n}+2^n$, $2^{n+\log n}\not\asymp2^n$.
 We write $f(n)\doteq g(n)$ (read $f(n)$ dot equals $g(n)$) if the coefficients of the dominant terms in the exponents of $f(n)$ and $g(n)$ match,
\[\lim_{n\to\infty}\frac{\log f(n)}{\log g(n)} = 1.\]
For instance, $2^{3n}\doteq2^{3n+n^{1/4}}$, $2^{2^n}\not\doteq2^{2^{n+\log n}}$. Note that $f(n)\asymp g(n)$ implies $f(n)\doteq g(n)$, but the converse is not true.

For any $q\in\bR_{>0}$, we write $\log_q(\cdot)$ for the logarithm to the base $q$. In particular, let $\log(\cdot)$ and $\ln(\cdot)$ denote logarithms to the base $2$ and $e$, respectively.

For any $\cA\subseteq\Omega$, the indicator function of $\cA$ is defined as, for any   $x\in\Omega$, 
\[\one_{\cA}(x)\coloneqq\begin{cases}1,&x\in \cA\\0,&x\notin \cA\end{cases}.\]
At times, we will slightly abuse notation by saying that $\one_{\sfA}$ is $1$ when event $\sfA$ happens and 0 otherwise. Note that $\one_{\cA}(\cdot)=\indicator{\cdot\in\cA}$.

\noindent\textbf{Sets.}
For any two nonempty sets $\cA$ and $\cB$ with addition and multiplication by a real scalar, let $\cA+\cB$ 
% and $\cA\cdot\cB$ 
denote the Minkowski sum 
% and Minkowski product 
of them which is defined as
$\cA+\cB\coloneqq\curbrkt{a+b\colon a\in\cA,b\in\cB}$.
% respectively.
If $\cA=\{x\}$ is a singleton set, we write $x+\cB$ and
 % $x\cdot\cB$ 
 for $\{x\}+\cB$.
  % and $\{x\}\cdot\cB$.
For any $ r\in\bR $, the $r$-dilation of $\cA$ is defined as $ r\cA \coloneqq \curbrkt{r\va:\va\in\cA} $. 
In particular, $ -\cA\coloneqq (-1)\cA $. 

For $M\in\bZ_{>0}$, we let $[M]$ denote the set of first $M$ positive integers $\{1,2,\cdots,M\}$.

\noindent\textbf{Geometry.}
Let $ \normtwo{\cdot} $ denote the Euclidean/$\ell_2$-norm. Specifically, for any $\vx\in\bR^n$,
\[ \normtwo{\vx} \coloneqq\paren{\sum_{i=1}^n\vx(i)^2}^{1/2}.\]

With slight abuse of notation, we let $|\cdot|$ denote the ``volume'' of a set w.r.t.\ a measure that is obvious from the context. 
If $ \cA $ is a finite set, then $ |\cA| $ denotes the cardinality of $\cA$ w.r.t.\ the counting measure. 
For a set $ \cA\subset\bR^n $, let
\begin{align}
\aff(\cA) &\coloneqq \curbrkt{\sum_{i = 1}^k\lambda_i\va_i:k\in\bZ_{\ge1};\;\forall i\in[k], \va_i\in\cA,\lambda_i\in\bR,\sum_{i = 1}^k\lambda_i = 1} \notag 
\end{align}
denote the \emph{affine hull} of $\cA$, i.e., the smallest affine subspace containing $\cA$. 
If $ \cA $ is a connected compact set in $ \bR^n $ with nonempty interior and $ \aff(\cA) = \bR^n $, then $ |\cA| $ denotes the volume of $\cA$ w.r.t.\ the $n$-dimensional Lebesgue measure. 
If $ \aff(\cA) $ is a $k$-dimensional affine subspace for $ 1\le k<n $, then $ |\cA| $ denotes the $k$-dimensional Lebesgue volume of $\cA$. 
% Specifically, for any Euclidean body $\cA\subseteq\bR^n$,
% \[\vol_n(\cA)=\int_\cA\diff \vx=\int_{\bR^n}\one_{\cA}(\vx)\diff\vx,\]
% where $\diff\vx$ denotes the differential of $\vx$ with respect to (w.r.t.) the Lebesgue measure on $\bR^n$. 
% When the dimension $n$ is obvious from the context, we will also use the shorthand notation $|\cdot|$ for $\vol_n(\cdot)$. 
% For convenience, the  subscript for dimension  will be dropped if no confusion will be caused. 
% If $\cA\subseteq\bR^n$ is an $n$-dimensional body with nonempty interior, we write $\vol(\cA)$ for $\vol_n(\cA)$; if $\cA\subseteq\bR^n$ is an $(n-1)$-dimensional hypersurface, we write $\area(\cA)$ for $\vol_{n-1}(\cA)$.

The closed $n$-dimensional Euclidean unit ball is defined as
\[\cB^n \coloneqq \curbrkt{\vy\in\bR^n\colon \normtwo{\vy} \le 1}.\]
The $(n-1)$-dimensional Euclidean unit sphere is defined as
\[\cS^{n-1} \coloneqq \curbrkt{\vy\in\bR^n\colon \normtwo{\vy} = 1}.\]
For any $ \vx\in\bR^n $ and $ r\in\bR_{>0} $, let $ \cB^n(r) \coloneqq r\cB^n, \cS^{n-1}(r) \coloneqq r\cS^{n-1} $ and $ \cB^n(\vx,r) \coloneqq \vx + r\cB^n, \cS^{n-1}(\vx,r) \coloneqq \vx + r\cS^{n-1} $. 
% We will drop the superscript for  dimension  when they are clear from the context.
% When the center of the ball or sphere is not important, we also drop the first argument. 

Let $ V_n \coloneqq |\cB^n| $.

\section{Preliminaries}
\label{sec:prelim}

\begin{lemma}[Change of variable]
\label{lem:change-var}
Let $ \cU\subset\bR^n $ be an open set and $ \phi\colon\cU\to\bR^n $ be an injective differentiable function with continuous partial derivatives, the Jacobian of which is nonzero for every $ \vx\in\cU $.
Then for any compactly supported, continuous function $ f \colon \phi(\cU)\to\bR $, the substitution $ \vv = \phi(\vu) $ yields the following formula
\begin{align}
\int_{\phi(\cU)} f(\vv) \diff\vv &= \int_\cU f(\phi(\vu)) \abs{\det(\nabla\phi)(\vu)} \diff \vu, \notag 
\end{align}
where $ \nabla\phi\in\bR^{n\times n} $ denotes the Jacobian matrix of $\phi$. 
\end{lemma}

\begin{theorem}[Laplace's method]
\label{thm:lap}
Suppose $ f\colon\bR^d\to\bR $ is a twice continuously differentiable function on $ \cA\subset\bR^d $, and there exists a unique point $ \vec t^*\in\interior(\cA) $ (where $ \interior(\cdot) $ denotes the interior of a set) such that 
\begin{align}
f(\vec t^*) = \min_{\vec t\in\cA} f(\vec t),\quad 
(\hess f)(\vec t^*) \succ 0, \notag 
\end{align}
where $ \hess f\in\bR^{d\times d} $ denotes the Hessian matrix of $f$. 
Suppose $ g(\vec t) $ is positive. 
Then 
\begin{align}
\int_{\cA} g(\vec t)e^{-Mf(\vec t)} \diff\vec t 
& \stackrel{M\to\infty}{\asymp} e^{-Mf(\vec t^*)}\paren{\frac{2\pi}{M}}^{d/2} \frac{g(\vec t^*)}{\sqrt{\det((\hess f)(\vec t^*))}} . \notag 
\end{align}
\end{theorem}

\begin{theorem}[\cramer]
\label{thm:cramer-ldp}
Let $ \curbrkt{\bfx_i}_{i = 1}^n $ be a sequence of i.i.d.\ real-valued random variables. 
Let $ \bfs_n \coloneq\frac{1}{n}\sum_{i = 1}^n\bfx_i $. 
Then for any closed $ \cF\subset\bR $,
\begin{align}
\limsup_{n\to\infty} \frac{1}{n}\ln\prob{\bfs_n\in \cF} &\le -\inf_{x\in \cF} \sup_{\lambda\in\bR}\curbrkt{ \lambda x - \ln\expt{e^{\lambda \bfx_1}} }; \notag 
\end{align}
and for any open $ \cG\subset\bR $, 
\begin{align}
\liminf_{n\to\infty} \frac{1}{n}\ln\prob{\bfs_n\in \cG} &\ge -\inf_{x\in \cG} \sup_{\lambda\in\bR}\curbrkt{\lambda x - \ln\expt{e^{\lambda \bfx_1}}}. \notag 
\end{align}
Furthermore, when $\cF$ or $\cG$ corresponds to the upper (resp. lower) tail of $ \bfs_n $, the maximizer $ \lambda\ge0 $ (resp. $ \lambda\le0 $).
\end{theorem}

\section{Basic definitions and facts}
\label{sec:def}

Given the intimate connection between packing and error-correcting codes, we will interchangeably use the terms ``multiple packing'' and ``list-decodable code''. 
The parameter $ L\in\bZ_{\ge2} $ is called the \emph{multiplicity of overlap} or the \emph{list-size}. 
The parameter $N$ % and $P$ (in the case of bounded packing) 
is called the \emph{noise power constraints}. 
Elements of a packing are called either \emph{points} or \emph{codewords}. 
We will call a size-$L$ subset of a packing an \emph{$L$-list}. 
This paper is only concerned with the fundamental limits of multiple packing for asymptotically large dimension $n\to\infty$. 
When we say ``a'' code $ \cC $, we always mean an infinite sequence of codes $ \curbrkt{\cC_i}_{i\ge1} $ where $ \cC_i\subset\bR^{n_i} $ and $ \curbrkt{n_i}_{i\ge1} $ is an increasing sequence of positive integers. 
% We call $\cC$ a \emph{spherical code} if $ \cC\subset\cS^{n-1}(\sqrt{nP}) $ and we call it a \emph{ball code} if $ \cC\subset\cB^n(\sqrt{nP}) $. 

In the rest of this section, we list a sequence of formal definitions and some facts associated with these definitions. 

% \begin{definition}[Bounded multiple packing]
% \label{def:packing-ball}
% Let $ N,P>0 $ and $ L\in\bZ_{\ge2} $. 
% A subset $ \cC \subseteq \cB^n( \sqrt{nP}) $ is called a \emph{$ (P,N,L-1) $-list-decodable code} (a.k.a.\ a \emph{$(P,N,L-1)$-multiple packing}) if for every $ \vy\in\bR^n $,
% \begin{align}
% \card{\cC\cap\cB^n(\vy, \sqrt{nN})} \le L-1 .
% \label{eqn:packing-ball}
% \end{align}
% The \emph{rate} (a.k.a.\ \emph{density}) of $ \cC $ is defined as
% \begin{align}
% R(\cC) \coloneqq \frac{1}{n}{\ln\cardC} .
% \label{eqn:density-bounded}
% \end{align}
% \end{definition}

\begin{definition}[Multiple packing]
\label{def:packing-euclidean}
Let $ N>0 $ and $ L\in\bZ_{\ge2} $. 
A subset $ \cC \subseteq \bR^n $ is called a \emph{$ (N,L-1) $-list-decodable code} (a.k.a.\ an \emph{$(N,L-1)$-multiple packing}) if for every $ \vy\in\bR^n $,
\begin{align}
\card{\cC\cap\cB^n(\vy, \sqrt{nN})} \le L-1 .
\label{eqn:packing-ball-repeat}
\end{align}
The \emph{rate} (a.k.a.\ \emph{density}) of $ \cC $ is defined as 
\begin{align}
R(\cC) \coloneqq& \limsup_{K\to\infty} \frac{1}{n}\ln\frac{\card{\cC\cap (K\cB)}}{\card{K\cB}}, 
\label{eqn:density-unbounded}
\end{align}
where $ \cB $ is an arbitrary centrally symmetric connected compact set in $ \bR^n $ with nonempty interior. 
\end{definition}

\begin{remark}
Common choices of $ \cB $ include the unit ball $ \cB^n $, the unit cube $ [-1,1]^n $, the fundamental Voronoi region $ \cV_\Lambda $ of a (full-rank) lattice $ \Lambda\subset\bR^n $, etc. 
Some choices of $\cB$ may be more convenient than the others for analyzing certain ensembles of packings. 
Therefore, we do not fix the choice of $\cB$ in \Cref{def:packing-euclidean}. 
\end{remark}

\begin{remark}
% It is a slight abuse of notation to write $ R(\cC) $ to refer to the rate of either a bounded packing or an unbounded packing. 
% However, the meaning of $ R(\cC) $ will be clear from the context. 
The rate of a packing (as per \Cref{eqn:density-unbounded}) is also called the \emph{normalized logarithmic density} in the literature. 
It measures the the normalized number of points per unit volume. 
\end{remark}

Note that the condition given by \Cref{eqn:packing-ball-repeat} is equivalent to that for any $ (\vx_1,\cdots,\vx_L)\in\binom{\cC}{L} $, 
\begin{align}
\bigcap_{i = 1}^L\cB^n(\vx_i,\sqrt{nN}) = \emptyset. \label{eqn:packing-ball-alternative} 
\end{align}

\begin{definition}[Chebyshev radius and average squared radius of a list]
\label{def:cheb-rad-avg-rad}
Let $ \vx_1,\cdots,\vx_L $ be $L$ points in $ \bR^n $. 
Then the \emph{squared Chebyshev radius} $ \rad^2(\vx_1,\cdots,\vx_L) $ of $ \vx_1,\cdots,\vx_L $ is defined as the (squared) radius of the smallest ball containing $ \vx_1,\cdots,\vx_L $, i.e., 
\begin{align}
\rad^2(\vx_1,\cdots,\vx_L) \coloneqq& \min_{\vy\in\bR^n} \max_{i\in[L]} \normtwo{\vx_i - \vy}^2. \label{eqn:cheb-rad} 
\end{align}
The \emph{average squared radius} $ \ol\rad^2(\vx_1,\cdots,\vx_L) $ of $ \vx_1,\cdots,\vx_L $ is defined as the average squared distance to the centroid, i.e., 
\begin{align}
\ol\rad^2(\vx_1,\cdots,\vx_L) \coloneqq& \frac{1}{L}\sum_{i = 1}^L{\normtwo{\vx_{i} - \vx}^2}, \label{eqn:avg-rad}
\end{align}
where $ \vx \coloneqq \frac{1}{L}\sum_{i = 1}^L\vx_i $ is the centroid of $ \vx_1,\cdots,\vx_L $. We refer to the square root of the average squared radius as the average radius of the list.
\end{definition}

% \begin{remark}
% Note that the condition given by \Cref{eqn:packing-ball,eqn:packing-ball-repeat} is equivalent to that for any $ \vx_1,\cdots,\vx_L\in\cC $, $ \rad^2(\vx_1,\cdots,\vx_L)>nN $. 
% \end{remark}

\begin{remark}
One should note that for an $L$-list $\cL$ of points, the smallest ball containing $\cL$ is not necessarily the same as the \emph{circumscribed ball}, i.e., the ball such that all points in $\cL$ live on the boundary of the ball. 
The circumscribed ball of the polytope $ \conv\curbrkt{\cL} $ spanned by the points in $\cL$ may not exist. 
If it does exist, it is not necessarily the smallest one containing $ \cL $. 
However, whenever it exists, the smallest ball containing $\cL$ must be the circumscribed ball of a certain \emph{subset} of $\cL$. 
\end{remark}

\begin{remark}
\label{rk:motivation-avg-rad}
We remark that the motivation behind the definition of average squared radius (\Cref{eqn:avg-rad}) is to replace the maximization in \Cref{eqn:cheb-rad} with average. 
\begin{align}
\min_{\vy\in\bR^n} \exptover{\bfi\sim[L]}{\normtwo{\vx_\bfi - \vy}^2} 
&= \min_{\vy\in\bR^n} \frac{1}{L}\sum_{i = 1}^L \sum_{j = 1}^n \paren{\vx_i(j) - \vy(j)}^2 \notag\\
&= \min_{(y_1,\cdots,y_n)\in\bR^n} \sum_{j = 1}^n \frac{1}{L}\sum_{i = 1}^L\paren{\vx_i(j) - y_j}^2 \label{eqn:before-interchange} \\
&= \sum_{j = 1}^n \min_{y_j\in\bR} \frac{1}{L}\sum_{i = 1}^L\paren{\vx_i(j) - y_j}^2 \label{eqn:interchange} \\
&= \frac{1}{L}\sum_{i = 1}^L\sum_{j = 1}^n\paren{\vx_i(j) - y_j^*}^2 \label{eqn:opt-y} \\
&= \frac{1}{L} \sum_{i = 1}^L \normtwo{\vx_i - \vx}^2. \label{eqn:relax-formula}
\end{align}
\Cref{eqn:interchange} holds since the inner summation $ \frac{1}{L}\sum_{i = 1}^L \paren{\vx_i(j) - \vy(j)}^2 $ in \Cref{eqn:before-interchange} only depends on $ y_j $ among all $ y_1,\cdots,y_n $. 
\Cref{eqn:opt-y} follows since for each $j$, the minimizer of the minimization in \Cref{eqn:interchange} is $ y_j^* \coloneqq \frac{1}{L}\sum_{i = 1}^L\vx_i(j) $. 
In \Cref{eqn:relax-formula}, the minimizer $ \vy^* $ equals $ \vx \coloneqq \frac{1}{L}\sum_{i = 1}^L\vx_i $. 
\end{remark}

\begin{definition}[Chebyshev radius and average squared radius of a code]
\label{def:cheb-rad-avg-rad-code}
Given a code $ \cC\subset\bR^n $ of rate $R$, the squared \emph{$(L-1)$-list-decoding radius} of $\cC$ is defined as
\begin{align}
\rad^2_L(\cC) \coloneqq&  \min_{\cL\in\binom{\cC}{L}} \rad^2(\cL). \label{eqn:rad-code}
\end{align}
The \emph{$(L-1)$-average squared radius} of $\cC$ is defined as 
\begin{align}
\ol\rad^2_L(\cC) \coloneqq& \min_{\cL\in\binom{\cC}{L}} \ol\rad^2(\cL). \label{eqn:avgrad-code}
\end{align}
\end{definition}

% \begin{definition}[Bounded average-radius multiple packing]
% \label{def:avg-rad-multi-pack-bounded}
% A subset $ \cC\subset\cB^n(\sqrt{nP}) $ is called a \emph{$ (P,N,L-1) $-average-radius list-decodable code} (a.k.a.\ a \emph{$ (P,N,L-1) $-average-radius multiple packing}) if $ \ol\rad^2_L(\cC)>nN $. 
% The \emph{rate} (a.k.a.\ \emph{density}) $ R(\cC) $ of $ \cC $ is given by \Cref{eqn:density-bounded}. 
% The \emph{$(P,N,L-1)$-average-radius list-decoding capacity} (a.k.a.\ \emph{$(P,N,L-1)$ average-radius multiple packing density}) is defined as
% $$ \ol C_{L-1}(P,N) \coloneqq \limsup_{n\to\infty} \limsup_{\cC \subseteq \cB^n( \sqrt{nP})\colon \ol\rad^2_L(\cC)>nN} R(\cC) .$$ 
% The \emph{$(L-1)$-average-radius list-decoding radius} at rate $R$ with input constraint $P$ is defined as
% \begin{align}
% \ol\rad^2_L(P,R) \coloneqq& \limsup_{n\to\infty} \limsup_{\cC \subseteq \cB^n( \sqrt{nP}) \colon R(\cC)\ge R} \ol\rad^2_L(\cC). \notag 
% \end{align}
% \end{definition}

\begin{definition}[Average-radius multiple packing]
\label{def:avg-rad-multi-pack-unbounded}
A subset $ \cC\subset\bR^n $ is called an \emph{$ (N,L-1) $-average-radius list-decodable code} (a.k.a.\ an \emph{$ (N,L-1) $-average-radius multiple packing}) if $ \ol\rad^2_L(\cC)>nN $. 
The \emph{rate} (a.k.a.\ \emph{density}) $ R(\cC) $ of $ \cC $ is given by \Cref{eqn:density-unbounded}. 
The \emph{$(N,L-1)$-average-radius list-decoding capacity} (a.k.a.\ \emph{$(N,L-1)$ average-radius multiple packing density}) is defined as
$$ \ol C_{L-1}(N) \coloneqq \limsup_{n\to\infty} \limsup_{\cC \subseteq \bR^n\colon \ol\rad^2_L(\cC)>nN} R(\cC) .$$ 
The \emph{squared $(L-1)$-average-radius list-decoding radius} at rate $R$ (without input constraint) is defined as
\begin{align} 
\ol\rad^2_L(R) \coloneqq& \limsup_{n\to\infty} \limsup_{\cC \subseteq \bR^n \colon R(\cC)\ge R} \ol\rad^2_L(\cC). \notag 
\end{align}
\end{definition}

Note that $ (L-1) $-list-decodability defined by \Cref{eqn:packing-ball-alternative} is equivalent to $ \rad^2_L(\cC) > nN$. 
We also define the \emph{$(N,L-1)$-list-decoding capacity} (a.k.a.\ \emph{$(N,L-1)$-multiple packing density}) $ C_{L-1}(N) $ and the \emph{$(L-1)$-list-decoding radius} $ \rad^2_L(R) $ at rate $R$: 
\begin{align}
C_{L-1}(N) &\coloneqq \limsup_{n\to\infty} \limsup_{\cC \subseteq \bR^n\colon \rad^2_L(\cC)>nN} R(\cC) , \notag \\
\rad^2_L(R) &\coloneqq \limsup_{n\to\infty} \limsup_{\cC \subseteq \bR^n \colon R(\cC)\ge R} \rad^2_L(\cC). \notag 
\end{align}

% We also define the \emph{$(P,N,L-1)$-list-decoding capacity} (a.k.a.\ \emph{$(P,N,L-1)$-multiple packing density})
% $$ C_{L-1}(P,N) \coloneqq \limsup_{n\to\infty} \limsup_{\cC \subseteq \cB^n( \sqrt{nP})\colon \rad^2_L(\cC)>nN} R(\cC) ,$$ 
% and the \emph{$(L-1)$-list-decoding radius} at rate $R$ with input constraint $P$
% \begin{align}
% \rad^2_L(P,R) \coloneqq& \limsup_{n\to\infty} \limsup_{\cC \subseteq \cB^n( \sqrt{nP}) \colon R(\cC)\ge R} \rad^2_L(\cC). \notag 
% \end{align}

Since the average radius is at most the Chebyshev radius, average-radius list-decodability is stronger than regular list-decodability. 
Any lower (resp. upper) bound on $ \ol C_{L-1}(N) $ (resp. $ C_{L-1}(N) $) is automatically a lower (resp. upper) bound on $ C_{L-1}(N) $ (resp. $ \ol C_{L-1}(N) $). 
% The same relation also holds for the unbounded versions $ C_{L-1}(N) $ and $ \ol C_{L-1}(N) $. 
Proving upper/lower bounds on $ C_{L-1}(N) $ (resp. $ \ol C_{L-1}(N) $) is equivalent to proving upper/lower bounds on $ \rad^2_L(R) $ (resp. $ \ol\rad^2_L(R) $).
% The same relation also holds for the unbounded versions $ C_{L-1}(N) $ (resp. $ \ol C_{L-1}(N) $) and $ \rad^2_L(R) $ (resp. $ \ol\rad^2_L(R) $). 

\subsection{Different notions of density of packings}
\label{sec:comments-density}

We measure the density of 
% a bounded packing using \Cref{eqn:density-bounded} and measure that of 
a packing using \Cref{eqn:density-unbounded}. 
In the literature, there exists another commonly used notion of density for multiple packings. 
It counts the fraction of space occupied by the union of the balls of radius $ \sqrt{nN} $ centered around points in the packing. 
Specifically, for an $ (N,L-1) $-packing $ \cC\subset\bR^n $,  
\begin{align}
\Delta(\cC) \coloneqq& \limsup_{P\to\infty} \frac{1}{n} \ln \paren{ \card{\bigcup_{\vx\in\cC}\cB^n(\vx, \sqrt{nN})\cap\cB^n(\sqrt{nP})}\big/\card{\cB^n( \sqrt{nP})} }. \label{eqn:density-unbdd-frac-space}
\end{align}

We prove the following statement. 
\begin{theorem}
\label{thm:two-notions-unbdd-density}
Let $ N>0 $ and $ L\in\bZ_{\ge2} $. 
Let $ \cC\subset\bR^n $ be an $(N,L-1)$-multiple packing. 
Then $ \Delta(\cC) \stackrel{n\to\infty}{\asymp} R(\cC) + \frac{1}{2}\ln(2\pi eN) $. 
\end{theorem}

\begin{proof}
Note that for sufficiently large $P$, 
\begin{align}
\card{\bigcup_{\vx\in\cC}\cB^n(\vx, \sqrt{nN})\cap\cB^n(\sqrt{nP})}
&\lesssim \card{\cC\cap\cB^n(\sqrt{nP})}\cdot\card{\cB^n(\sqrt{nN})}. \notag 
\end{align}
Therefore, 
\begin{align}
\Delta(\cC) &\le \limsup_{P\to\infty} \frac{1}{n} \ln\frac{\card{\cC\cap\cB^n(\sqrt{nP})}\cdot\card{\cB^n(\sqrt{nN})}}{\card{\cB^n( \sqrt{nP})}} \notag \\
&= R(\cC) + \frac{1}{n}\ln\card{\cB^n( \sqrt{nN})} \notag \\
&\asymp R(\cC) + \frac{1}{n}\ln\paren{\sqrt{nN}^n\cdot \frac{1}{\sqrt{\pi n}}\sqrt{\frac{2\pi e}{n}}^n } \notag \\
&\asymp R(\cC) + \frac{1}{2}\ln(2\pi eN). \label{eqn:density-ub} 
\end{align}

On the other hand, we claim %\footnote{
% 	To lower bound $ \card{\bigcup\limits_{\vx\in\cC}\cB^n(\vx,\sqrt{nN})\cap\cB^n(\sqrt{nP})} $, consider $ \curbrkt{\cB^n(\vx,\sqrt{nN})\cap\cB^n(\sqrt{nP}):\vx\in\cC} $. 
% 	Imagine that we rearrange $ \vx\in\cC $ so that $ \cB^n(\vx,\sqrt{nN})\cap\cB^n(\sqrt{nP}) $ for different $ \vx\in\cC $ becomes disjoint. 
% 	Note that the volume induced by the rearranged packing is basically $ \card{\cC\cap\cB^n(\sqrt{nP})}\cdot\card{\cB^n(\sqrt{nN})} $. 
% 	On the other hand, during this rearranging process, each point in $ \cB^n(\vx,\sqrt{nN}) $ (for any $ \vx\in\cC $) spits out at most $L-2$ copies of itself since by $(L-1)$-list-decodability, each point in a ball is covered by at most $L-2$ other balls. 
% 	Therefore, $ \card{\cC\cap\cB^n(\sqrt{nP})}\cdot\card{\cB^n(\sqrt{nN})} \lesssim (L-1)\card{\bigcup\limits_{\vx\in\cC}\cB^n(\vx,\sqrt{nN})\cap\cB^n(\sqrt{nP})} $, which justifies \Cref{eqn:density-ub}. 
% } 
that for sufficiently large $P$, 
\begin{align}
\card{\bigcup_{\vx\in\cC}\cB^n(\vx, \sqrt{nN})\cap\cB^n(\sqrt{nP})} &\gtrsim \frac{1}{L-1} \card{\cC\cap\cB^n(\sqrt{nP})}\cdot\card{\cB^n(\sqrt{nN})}. \label{eqn:vol-relation} 
\end{align}
The above claim is justified at the end of this subsection. 
Then following the same lines of calculations above, we have
\begin{align}
\Delta(\cC) &\ge R(\cC) + \frac{1}{2}\ln(2\pi e N) - o(1). \label{eqn:density-lb} 
\end{align}
Combining \Cref{eqn:density-ub,eqn:density-lb}, we have that for any $ (N,L-1) $-packing $\cC$, 
\begin{align}
\Delta(\cC) &\asymp R(\cC) + \frac{1}{2}\ln(2\pi eN). \notag 
\end{align}

To see \Cref{eqn:vol-relation}, consider the set of convex sets 
\begin{align}
\sD&\coloneqq\curbrkt{\cD(\vx):\vx\in\cC,\;\cB^n(\vx,\sqrt{nN})\cap\cB^n(\sqrt{nP})\ne\emptyset }, \notag
\end{align}
where 
\begin{align}
\cD(\vx)&\coloneqq \cB^n(\vx,\sqrt{nN})\cap\cB^n(\sqrt{nP}). \notag 
\end{align}
Note that some elements in $\sD$ are balls $ \cB^n(\vx,\sqrt{nN}) $, while other elements are the intersection of two balls $ \cB^n(\vx,\sqrt{nN})\cap\cB^n(\sqrt{nP}) $ for some $ \vx\in\cC $. 

We now rearrange all $ \cD(\vx) $ in $ \sD $ so that they become disjoint. 
Then the volume induced by the resulting packing is 
\begin{align}
\sum_{\cD\in\sD}\card{\cD} &= \sum_{\substack{\vx\in\cC\\\cB^n(\vx,\sqrt{nN})\cap\cB^n(\sqrt{nP})\ne\emptyset}} |\cD(\vx)| \notag \\
&\approx \sum_{\vx\in\cC\cap\cB^n(\sqrt{nP})} \card{\cD(\vx)} \notag \\
&= \sum_{\vx\in\cC\cap\cB^n(\sqrt{nP})} \card{\cB^n(\vx,\sqrt{nN})\cap\cB^n(\sqrt{nP})}. \label{eqn:vol-relation-1}
\end{align}
On the other hand, for any $ \vy\in\bigcup\limits_{\cD\in\sD}\cD $, 
it can split into at most $L-1$ points after the above process since by $(L-1)$-list-decodability, each point in a ball is covered by at most $L-2$ other balls. 
Therefore, the volume induced by the rearranged packing is at most 
\begin{align}
(L-1)\card{\bigcup_{\cD\in\sD} \cD} &= (L-1)\card{\bigcup_{\vx\in\cC}\cB^n(\vx,\sqrt{nN})\cap\cB^n(\sqrt{nP})}. \label{eqn:vol-relation-2} 
\end{align}
Combining \Cref{eqn:vol-relation-1,eqn:vol-relation-2}, we have, for sufficiently large $P$,
\begin{align}
\card{\bigcup_{\vx\in\cC}\cB^n(\vx,\sqrt{nN})\cap\cB^n(\sqrt{nP})} &\gtrsim \frac{1}{L-1} \sum_{\vx\in\cC\cap\cB^n(\sqrt{nP})} \card{\cB^n(\vx,\sqrt{nN})\cap\cB^n(\sqrt{nP})} \notag \\
&\approx \frac{1}{L-1}\card{\cC\cap\cB^n(\sqrt{nP})}\card{\cB^n(\sqrt{nN})}, \notag
\end{align}
as claimed in \Cref{eqn:vol-relation}. 
\end{proof}

\subsection{Chebyshev radius and average radius}
\label{sec:comments-radius}

In this section, we present several different representations of the Chebyshev radius and average squared radius. 
Some of them will be crucially used in the subsequent sections of this paper. 
These representations are summarized in the following theorem which will be proved in the subsequent subsections. 
\begin{theorem}
\label{thm:repr-rad}
Let $ L\in\bZ_{\ge2} $ and $ \vx_1,\cdots,\vx_L\in\bR^n $. 
Then the squared Chebyshev radius of $ \vx_1,\cdots,\vx_L $ admits the following alternative representations: 
\begin{enumerate}
\item 
$\begin{aligned}
\rad^2(\vx_1,\cdots,\vx_L) &= \max_{{\vec z}\in\Delta_L} \sum_{i = 1}^L{\vec z}(i)\normtwo{\vx_i - \vy_{\vec z}}^2, \notag 
\end{aligned}$
where $ \vy_{\vec z} \coloneqq \sum_{i = 1}^L{\vec z}(i)\vx_i $ and $ \Delta_L $ denotes the probability simplex on $ [L] $;

\item 
$\begin{aligned}
\rad^2(\vx_1,\cdots,\vx_L) &= \lim_{p\to\infty} \rad^{(p)}(\vx_1,\cdots,\vx_L), \notag 
\end{aligned}$
where 
\begin{align}
\rad^{(p)}(\vx_1,\cdots,\vx_L) &\coloneqq \paren{\min_{\vy\in\bR^n}\frac{1}{L}\sum_{i = 1}^L\normtwo{\vx_i - \vy}^{2p}}^{1/p}; \notag 
\end{align} 

\item 
there exists a unique $ 1\le q\le\infty $ depending on $ \vx_1,\cdots,\vx_L $ such that 
\begin{align}
\rad^2(\vx_1,\cdots,\vx_L) &= \paren{\frac{1}{L}\sum_{i = 1}^L\normtwo{\vx_i - \vx}^{2q}}^{1/q}, \notag 
\end{align}
and $ \vx\coloneqq\frac{1}{L}\sum_{i = 1}^L\vx_i $. 
\end{enumerate}

The average squared radius of $ \vx_1,\cdots,\vx_L $ admits the following alternative representations:
\begin{enumerate}
\item 
$\begin{aligned}
\ol\rad^2(\vx_1,\cdots,\vx_L) &= \frac{1}{L}\sum_{i = 1}^L\normtwo{\vx_i}^2 - \normtwo{\vx}^2; \notag 
\end{aligned}$

\item 
$\begin{aligned}
\ol\rad^2(\vx_1,\cdots,\vx_L) 
&= \frac{L-1}{L^2}\sum_{i = 1}^L\normtwo{\vx_i}^2 - \frac{1}{L^2}\sum_{(i,j)\in[L]^2:i\ne j} \inprod{\vx_i}{\vx_j}; \notag 
\end{aligned}$

\item 
$\begin{aligned}
\ol\rad^2(\vx_1,\cdots,\vx_L)
&= \frac{1}{2L^2}\sum_{(i,j)\in[L]^2:i\ne j} \normtwo{\vx_i - \vx_j}^2. \notag 
\end{aligned}$
\end{enumerate}
\end{theorem}

\subsubsection{Another representation of the Chebyshev radius}
\label{sec:another-cheb-rad}
The Chebyshev radius involves a minimax expression which is in general tricky to handle. 
One can use minimax theorem to interchange the min and max and then compute the inner min explicitly. 
\begin{align}
\rad^2(\vx_1,\cdots,\vx_L) &= \min_{\vy\in\bR^n} \max_{i\in[L]} \normtwo{\vx_i - \vy}^2 
= \min_{\vy\in\bR^n} \max_{{\vec z}\in\Delta_L} \sum_{i = 1}^L {\vec z}(i) \normtwo{\vx_i - \vy}^2. \notag
\end{align}
The last equality follows since the maximum is always achieved by a singleton ${\vec z}\in\zo^L $. 
Note that the objective function on the RHS is linear (hence concave) in $ {\vec z} $ and quadratic (hence convex) in $\vy$. 
Therefore the max and min can be interchanged and we get
\begin{align}
\max_{{\vec z}\in\Delta_L} \min_{\vy\in\bR^n} \sum_{i = 1}^L{\vec z}(i) \normtwo{\vx_i - \vy}^2 
= \max_{{\vec z}\in\Delta_L} \min_{\vy\in\bR^n} \sum_{j = 1}^n\sum_{i = 1}^L{\vec z}(i)\paren{\vx_i(j) - \vy(j)}^2 
= \max_{{\vec z}\in\Delta_L} \sum_{j = 1}^n \min_{y_j\in\bR} \sum_{i = 1}^L{\vec z}(i)\paren{\vx_i(j) - y_j}^2 . \notag 
\end{align}
The last equality follows since each inner summation $ \sum_{i = 1}^L{\vec z}(i)\paren{\vx_i(j) - \vy(j)}^2 $ only depends on $ \vy(j) $ among all $ \vy(1),\cdots,\vy(n) $. 
For each $j$, the minimizing $ y_j^* $ equals
\begin{align}
y_j^* &\coloneqq \frac{\sum_{i = 1}^L{\vec z}(i)\vx_i(j)}{\sum_{i=1}^L{\vec z}(i)} = \sum_{i = 1}^L{\vec z}(i)\vx_i(j), \notag 
\end{align}
where the last equality is because $ {\vec z}\in\Delta_L $. 
Therefore 
\begin{align}
\rad^2(\vx_1,\cdots,\vx_L) &= \max_{{\vec z}\in\Delta_L} \sum_{i = 1}^L{\vec z}(i)\normtwo{\vx_i - \vy_{\vec z}}^2, \label{eqn:cheb-rad-interchange}
\end{align}
where $ \vy_{\vec z} \coloneqq \sum_{i = 1}^L{\vec z}(i)\vx_i $.

\subsubsection{Higher-order approximations to the Chebyshev radius}
\label{sec:higher-order-apx-cheb-rad}
As explained in \Cref{rk:motivation-avg-rad}, the average squared radius is a linear relaxation of the squared Chebyshev radius:
\begin{align}
\ol\rad^2(\vx_1,\cdots,\vx_L) &= \min_{\vy\in\bR^n}\exptover{\bfi\sim[L]}{\normtwo{\vx_\bfi - \vy}^2}. \label{eqn:recall-avg-rad-motivation} 
\end{align}
One may obtain better and better approximations to the squared Chebyshev radius by taking higher and higher order relaxations:
\begin{align}
\rad^{(p)}(\vx_1,\cdots,\vx_L) &= \min_{\vy\in\bR^n}\paren{\exptover{\bfi\sim[L]}{\normtwo{\vx_\bfi - \vy}^{2p}}}^{1/p} = \paren{\min_{\vy\in\bR^n}\exptover{\bfi\sim[L]}{\normtwo{\vx_\bfi - \vy}^{2p}}}^{1/p}, \label{eqn:rad-p}
\end{align}
where $ p\ge1 $. 
The second equality in \Cref{eqn:rad-p} follows since the $ f(\cdot) =  (\cdot)^{1/p} $ is monotonically increasing. 
Note that $ \rad^{(1)}(\vx_1,\cdots,\vx_L) = \ol\rad^2(\vx_1,\cdots,\vx_L) $. 
Moreover, since $ \expt{(\cdot)^p}^{1/p} $ is increasing in $p$, we have 
\begin{align}
\rad^{(p)}(\vx_1,\cdots,\vx_L)\xrightarrow{p\to\infty}\rad^2(\vx_1,\cdots,\vx_L). \notag
\end{align}
However, we do not know how to analyze $ \rad^{(p)} $. 
It seems difficulty to get a closed-form solution of the minimization since the minimizer $ \vy^* $ cannot be obtained by minimizing over $ \vy(j) $ for different $ j\in[n] $ separately.  

\subsubsection{More representations of the average squared radius}
\label{sec:more-repr-avg-rad}
Recall that \Cref{eqn:recall-avg-rad-motivation}, as a lower bound on $ \rad^2(\vx_1,\cdots,\vx_L) $, admits an explicit formula given by \Cref{eqn:relax-formula}:
\begin{align}
\rad^2(\vx_1,\cdots,\vx_L) \ge \ol\rad^2(\vx_1,\cdots,\vx_L) = \frac{1}{L}\sum_{i = 1}^L\normtwo{\vx_i - \vx}^2, \label{eqn:cheb-rad-lb}
\end{align}
where $ \vx\coloneqq\frac{1}{L}\sum_{i = 1}^L\vx_i $ denotes the centroid of $ \vx_1,\cdots,\vx_L $. 
On the other hand, we have
\begin{align}
\rad^2(\vx_1,\cdots,\vx_L) = \min_{\vy\in\bR^n}\max_{i\in[L]} \normtwo{\vx_i - \vy}^2 \le \max_{i\in[L]} \normtwo{\vx_i - \vx}^2 \eqqcolon \rad^2_{\max}(\vx_1,\cdots,\vx_L). \label{eqn:cheb-rad-ub} 
\end{align}
Contrasting \Cref{eqn:cheb-rad-lb,eqn:cheb-rad-ub}, by monotonicity and continuity of $ \norm{p}{\cdot} $ in $p$, we know that there exists $ 1\le p \le\infty $ such that 
\begin{align}
\rad^2(\vx_1,\cdots,\vx_L) = \paren{\frac{1}{L}\sum_{i = 1}^L\normtwo{\vx_i - \vx}^{2p}}^{1/p}.  \notag 
\end{align}
However, we do not know how to use the above observation for the following two reasons. 
Firstly, the above expression seems tricky to handle.
Secondly and more importantly, the number $p$ depends on $ \vx_1,\cdots,\vx_L $ and is typically different for different lists.

Finally, we provide several alternative expressions for $ \ol\rad^2(\vx_1,\cdots,\vx_L) $ which will be useful in the proceeding sections of this paper. 
\begin{align}
\ol\rad^2(\vx_1,\cdots,\vx_L) &= \frac{1}{L}\sum_{i = 1}^L\normtwo{\vx_i - \vx}^2 \notag \\
&= \frac{1}{L}\sum_{i = 1}^L\inprod{\vx_i - \vx}{\vx_i - \vx} \notag \\
&= \frac{1}{L}\sum_{i = 1}^L\paren{ \normtwo{\vx_i}^2 - 2\inprod{\vx_i}{\vx} + \normtwo{\vx}^2 } \notag \\
&= \frac{1}{L}\sum_{i = 1}^L\normtwo{\vx_i}^2 - 2\inprod{\vx}{\vx} + \normtwo{\vx}^2 \notag \\
&= \frac{1}{L}\sum_{i = 1}^L\normtwo{\vx_i}^2 - \normtwo{\vx}^2. \label{eqn:avg-rad-alternative} 
\end{align}

The above expression can be further written as 
\begin{align}
\ol\rad^2(\vx_1,\cdots,\vx_L) &= \frac{1}{L}\sum_{i = 1}^L\normtwo{\vx_i}^2 - \normtwo{\vx}^2 \notag \\
&= \frac{1}{L}\sum_{i = 1}^L\normtwo{\vx_i}^2 - \frac{1}{L^2}\sum_{(i,j)\in[L]^2} \inprod{\vx_i}{\vx_j} \notag \\
&= \frac{1}{L}\sum_{i = 1}^L\normtwo{\vx_i}^2 - \frac{1}{L^2}\sum_{i = 1}^L\normtwo{\vx_i}^2 - \frac{1}{L^2}\sum_{(i,j)\in[L]^2:i\ne j} \inprod{\vx_i}{\vx_j} \notag \\
&= \frac{L-1}{L^2}\sum_{i = 1}^L\normtwo{\vx_i}^2 - \frac{1}{L^2}\sum_{(i,j)\in[L]^2:i\ne j} \inprod{\vx_i}{\vx_j}. \label{eqn:avg-rad-formula-pw-corr} 
\end{align}

At last, \Cref{eqn:avg-rad-formula-pw-corr} can in turn be rewritten as 
\begin{align}
\ol\rad^2(\vx_1,\cdots,\vx_L) 
&= \frac{L-1}{L^2}\sum_{i = 1}^L\normtwo{\vx_i}^2 - \frac{1}{L^2}\sum_{(i,j)\in[L]^2:i\ne j} \inprod{\vx_i}{\vx_j} \notag \\
&= \frac{L-1}{L^2}\sum_{i = 1}^L\normtwo{\vx_i}^2 - \frac{1}{2L^2} \sum_{(i,j)\in[L]^2:i\ne j}\paren{\normtwo{\vx_i}^2 + \normtwo{\vx_j}^2} + \frac{1}{2L^2} \sum_{(i,j)\in[L]^2:i\ne j} \paren{\normtwo{\vx_i}^2 + \normtwo{\vx_j}^2 - 2\inprod{\vx_i}{\vx_j}} \notag \\
&= \frac{L-1}{L^2}\sum_{i = 1}^L\normtwo{\vx_i}^2 - \frac{1}{L^2} \sum_{(i,j)\in[L]^2:i\ne j}\normtwo{\vx_i}^2 + \frac{1}{2L^2}\sum_{(i,j)\in[L]^2:i\ne j} \normtwo{\vx_i - \vx_j}^2 \notag \\
&= \frac{1}{2L^2}\sum_{(i,j)\in[L]^2:i\ne j} \normtwo{\vx_i - \vx_j}^2 . \label{eqn:avg-rad-pw-dist} 
\end{align}

If all $ \vx_1,\cdots,\vx_L $ have the same $ \ell_2 $ norm $ \sqrt{nP} $, then \Cref{eqn:avg-rad-alternative}
\begin{align}
\ol\rad^2(\vx_1,\cdots,\vx_L) = nP - \normtwo{\vx}^2. \label{eqn:avg-rad-spherical-formula}
\end{align}
and \Cref{eqn:avg-rad-formula-pw-corr} becomes
\begin{align}
\ol\rad^2(\vx_1,\cdots,\vx_L) = \frac{L-1}{L}nP - \frac{1}{L^2}\sum_{(i,j)\in[L]^2:i\ne j} \inprod{\vx_i}{\vx_j} . \label{eqn:avg-rad-spherical-formula-pw-corr} 
\end{align}

\section{Lower bounds 
% \svcomm{for unbounded packings} 
for unbounded packings}
\label{sec:lb-ppp}

% \svcomm{Since this deals with ICs, why not have this later, maybe just before the section on upper bounds? This would give continuity to the discussion on sphere and ball packings}

% A Poisson point process satisfies
% \begin{enumerate}
% 	\item the number of points in any set is a Poisson random variable;
% 	\item conditioned on the number of points in a set, all points are independent and uniformly distributed in the set. 
% \end{enumerate}

In this section, we analyze average-radius list-decodability of a class of regular infinite constellations obtained by expurgating and tiling a random code supported over an $n$-dimensional hypercube. Using this, we prove the following lower bound on the $ (N,L-1) $-average-radius list-decoding capacity of multiple packings. 
\begin{theorem}
\label{thm:lb-ppp}
For any $ N>0 $ and $ L\in\bZ_{\ge2} $, the $ (N,L-1) $-average-radius list-decoding capacity is at least 
\begin{align}
\ol C_{L-1}(N) &\ge \frac{1}{2}\ln\frac{L-1}{2\pi eNL} - \frac{1}{2(L-1)}\ln L. \label{eqn:lb-ppp} 
\end{align}
\end{theorem}

\begin{remark}
Note that the above bound (\Cref{eqn:lb-ppp}) approaches $ \frac{1}{2}\ln\frac{1}{2\pi eN} $ as $ L\to\infty $. 
The latter quantity is known to be the list-decoding capacity for asymptotically large $L$ (see \Cref{sec:listdec-cap-large}). 
On the extreme, when $L=2$, the above bound becomes $ \frac{1}{2}\ln\frac{1}{8\pi eN} $ which recovers the best known bound due to Minkowski \cite{minkowski-sphere-pack}. 
\end{remark}

To prove the above theorem, let $ R<\frac{1}{2}\ln\frac{L-1}{2\pi eNL} - \frac{1}{2(L-1)}\ln L $ and $ \lambda_n\doteq e^{nR} $. 
The exact choice of $ \lambda_n $ is given by \Cref{eqn:lambda-ppp}. 

To analyze average-radius list-decodability of $\cC$, we first construct an average-radius list-decodable code $\cC_K$ supported within $ \cA = \cI^n $ where $ \cI\coloneqq[-K,K] $ is a sufficiently large interval for some $ K>0 $. We later tile this codebook over $\bR^n$ to obtain an infinite constellation having the same average squared radius as the finite codebook.

The finite codebook $\cC_K$ is obtained by drawing $M\coloneqq \lambda_n|\cA|$ points independently and uniformly at random from $\cA$ and expurgating the resulting codebook.
%Let $ M $ be the number of points in $ \cA $, i.e., $ M\coloneqq|\cC\cap\cA| $. 
%By \Cref{itm:ppp-def-poisson} of \Cref{def:ppp}, $M\sim\pois(\lambda_n|\cA|)$. 
%Conditioned on $M$, 
Let $\cC_K'\coloneqq\{ \vbfx_1, \cdots,\vbfx_M \}$ denote the $M$ independent points uniformly distributed over $ \cA $.

\begin{lemma}\label{lemma:finitecodebook_avgrad}
  There exists a finite codebook $\cC_K$ supported over $\cA$, having minimum average squared radius at least $\sqrt{nN}$ and density
  \[
    \frac{1}{n}\ln \frac{|\cC_K|}{|\cA|} \geq \frac{1}{2}\ln\frac{L-1}{2\pi eNL} - \frac{1}{2(L-1)}\ln L +o(1).
  \]
\end{lemma}
%By \Cref{itm:ppp-prop-uniform} of \Cref{fact:ppp-properties}, these points are independent and uniformly distributed in $\cA$. 
%That is, $ \vbfx_i(j)\iid\unif(\cI) $ for each $ 1\le i\le M $ and $ 1\le j\le n $. 
% Let $ \cL\in\binom{[M]}{L} $. 
% WLOG $\cL = [L]$.
% Let $ \vbfx \coloneqq \frac{1}{L}\sum_{i = 1}^L\vbfx_i $ be the centroid of the list. 
% Given $M$, we would like to bound

The first step is to bound
\begin{align}
\prob{\ol\rad^2(\vbfx_1,\cdots,\vbfx_L)\le nN} . \label{eqn:avg-rad-bd} 
\end{align}
for every subset of $L$ codewords in $\cC_K'$.
In fact, we will prove the following lemma.
\begin{lemma}\label{lemma:avg_rad_exponent}
  For any $\vbfx_1,\cdots,\vbfx_L$ drawn independently and uniformly at random from $\cA$, we have
  \[
    \prob{\ol\rad^2(\vbfx_1,\cdots,\vbfx_L) \le nN} = e^{nE(K)+o(n)} ,
  \]
  where the $o(n)$ term is independent of $K$, and
  \[
E(K)\coloneqq \frac{L-1}{2}\ln\frac{L-1}{2\pi eNL} - \frac{1}{2}\ln L + (L-1)\ln (2K). 
  \]
\end{lemma}

First, we note that
\Cref{eqn:avg-rad-bd} can be alternatively written as 
\begin{align}
\prob{\ol\rad^2(\vbfx_1,\cdots,\vbfx_L)\le nN}
&= \prob{\frac{1}{L}\sum_{i = 1}^L\normtwo{\vbfx_i}^2 - \normtwo{\vbfx}^2\le nN} \label{eqn:apply-avg-rad-alternative} \\
&= \prob{\sum_{i = 1}^L\normtwo{\vbfx_i}^2 - L\normtwo{\vbfx}^2 \le LnN} \notag \\
&= \prob{\sum_{j = 1}^n\sum_{i = 1}^L\vbfx_i(j)^2 - L\sum_{j = 1}^n\paren{\frac{1}{L}\sum_{i = 1}^L\vbfx_i(j)}^2 \le LnN} \notag \\
&= \prob{ \sum_{j = 1}^n\paren{\sum_{i = 1}^L\vbfx_i(j)^2 - \frac{1}{L}\paren{\sum_{i = 1}^L\vbfx_i(j)}^2} \le LnN }, \label{eqn:to-use-quadratic-form}
\end{align}
where \Cref{eqn:apply-avg-rad-alternative} is by \Cref{eqn:avg-rad-alternative} and $ \vbfx \coloneqq \frac{1}{L}\sum_{i = 1}^L\vbfx_i $. 

Define 
\begin{align}
g(x_1, \cdots,x_L) \coloneqq& \sum_{i = 1}^Lx_i^2 - \frac{1}{L}\paren{\sum_{i = 1}^Lx_i}^2. \notag 
\end{align}
The above probability (\Cref{eqn:to-use-quadratic-form}) can be rewritten as 
\begin{align}
\prob{\sum_{j = 1}^ng(\vbfx_1(j), \cdots,\vbfx_L(j)) \le LnN} \label{eqn:avg-rad-bd-g}
\end{align}
% where $ \vbfx_i(j)\iid\cN(0,P') $ for each $i\in[L]$ and $j\in[n]$.
where $ \vbfx_i(j)\iid\unif(\cI) $ for each $i\in[L]$ and $j\in[n]$. 

We note that the function $ g(\vec t) $ is a quadratic form of $ \vec t\in\bR^L $. 
Indeed, 
\begin{align}
g(\vec t) &= \sum_{i = 1}^L\vec t(i)^2 - \frac{1}{L}\paren{\sum_{i = 1}^L\vec t(i)}^2 % \notag \\
= \paren{1 - \frac{1}{L}} \sum_{i = 1}^L\vec t(i)^2 - \frac{2}{L} \sum_{i,j\in[L]: i<j} \vec t(i)\vec t(j) % \notag \\
= \vec t^\top A\vec t, \label{eqn:def-g}
\end{align}
where 
\begin{align}
A \coloneqq& I_L - \frac{1}{L}J_L \in\bR^{L\times L} \notag
\end{align}
and $ I_L $ denotes the $ L\times L $ identity matrix and $ J_L $ denotes the $ L\times L $ all-one matrix. 
Therefore we can write \Cref{eqn:avg-rad-bd} as 
\begin{align}
\prob{\sum_{j = 1}^n\vec\bfx_j^\top A\vec\bfx_j \le LnN}, \label{eqn:ppp-to-bound-clean}
\end{align}
where $ \vec\bfx_j\coloneqq[\vbfx_1(j),\cdots,\vbfx_L(j)]\in\bR^L $ and $ \vbfx_i(j)\iid\unif(\cI) $ for each $ 1\le i\le L $ and $ 1\le j\le n $.

\subsection{Large deviation principle}
\label{sec:ldp}

Since $ \vec\bfx_j^\top A\vec\bfx_j $ is independent for each $ 1\le j\le n $, we can apply the large deviation principle (\Cref{thm:cramer-ldp}) to get the asymptotic behaviour of \Cref{eqn:ppp-to-bound-clean}.
Specifically, 
\begin{align}
& \frac{1}{n}\ln\Cref{eqn:ppp-to-bound-clean} \notag \\
&\xrightarrow{n\to\infty} - \max_{\lambda\le0} \curbrkt{ \lambda LN - \ln\expt{e^{\lambda \vec\bfx^\top A\vec\bfx}} } \notag \\
&= - \max_{\lambda\ge0} \curbrkt{ -\lambda LN - \ln\expt{e^{-\lambda \vec\bfx^\top A\vec\bfx}} }, \label{eqn:ldp-to-be-cont}
\end{align}
where $ \vec\bfx \sim\unif^\tl(\cI) $. 

We need to compute the following integral:
\begin{align}
\expt{e^{ - \lambda \vec\bfx^\top A\vec\bfx}} &= \frac{1}{(2K)^L} \int_{\cI^L} e^{-\lambda \vec x^\top A\vec x} \diff\vec x \notag \\
&= \frac{1}{(2K)^L} \int_{[-1,1]^L} e^{-\lambda K^2\vec t^\top A\vec t} K^L \diff\vec t \notag \\
&= \frac{1}{2^L} \int_{[-1,1]^L} e^{-K^2\lambda\vec t^\top A\vec t} \diff\vec t, \label{eqn:rewritten} 
\end{align}
where $ \lambda\ge0 $.

% Now the integral to be computed can be compactly written as
% \begin{align}
% \int_{\cI^L} e^{-\vec t^\top A\vec t} \diff\vec t. \notag
% \end{align}
Note that $A\in\bR^{L\times L} $ has rank $L-1$ and therefore is singular, unfortunately. 
In fact, $A$ has eigendecomposition 
$A = PDP^{-1} $
where
\begin{align}
P \coloneqq& \begin{bmatrix}[c|c|c|c|c]
-1&-1&\cdots&-1&1 \\
&&&1&1\\
&&\iddots&&\vdots\\
&1&&&1\\
1&&&&1
\end{bmatrix} \in\bR^{L\times L} \notag 
\end{align}
consists of the eigenvectors of $A$ as its columns and
\begin{align}
D\coloneqq& \begin{bmatrix}
1&&&\\
&\ddots&&\\
&&1&\\
&&&0
\end{bmatrix} \in\bR^{L\times L}
\notag
\end{align}
consists of the eigenvalues of $A$ as its diagonal entries.
However, $P$ is not orthogonal. 
One can orthogonalize it using the Gram--Schmidt process which gives us an orthogonal matrix $U \in\bR^{L\times L} $. 
We claim that 
% \begin{strip}
\begin{align}
U &= \begin{bmatrix}
-\frac{1}{\sqrt{1\times2}} & -\frac{1}{\sqrt{2\times3}} & -\frac{1}{\sqrt{3\times4}} & -\frac{1}{\sqrt{4\times5}} & \cdots & -\frac{1}{\sqrt{(L-2)\times(L-1)}} & -\frac{1}{\sqrt{(L-1)\times L}} & \frac{1}{\sqrt{L}} \\
&&&& &  & \sqrt{\frac{L-1}{L}} & \frac{1}{\sqrt{L}} \\
&&&&& \sqrt{\frac{L-2}{L-1}} & -\frac{1}{\sqrt{(L-1)\times L}} & \vdots \\
&&&& \iddots & -\frac{1}{\sqrt{(L-2)\times(L-1)}} & \vdots & \vdots \\
&&& \sqrt{\frac{4}{5}} & \vdots & \vdots & \vdots & \vdots \\
&& \sqrt{\frac{3}{4}} & -\frac{1}{\sqrt{4\times5}} & \vdots & \vdots & \vdots & \vdots \\
& \sqrt{\frac{2}{3}} & -\frac{1}{\sqrt{3\times4}} & \vdots & \vdots & \vdots & \vdots & \vdots \\
\sqrt{\frac{1}{2}} & -\frac{1}{\sqrt{2\times3}} & -\frac{1}{\sqrt{3\times4}} & -\frac{1}{\sqrt{4\times5}} & \vdots & -\frac{1}{\sqrt{(L-2)\times(L-1)}} & -\frac{1}{\sqrt{(L-1)\times L}} & \frac{1}{\sqrt{L}}
\end{bmatrix} \in \bR^{L\times L}
. \label{eqn:matrix-u}
\end{align}
% \end{strip}
The above $U$ gives us the Singular Value Decomposition of $A$ which is $ A = UDU^\top $.
Note that $ U^\top = U^{-1} $ by orthogonality of $U$ and the diagonalization of $A$ is given by $ D = U^{-1}AU^{-\top} = U^\top A U $. 
Under the change of variable $ \vec t = U\vec y $, the quadratic form $ \vec t^\top A\vec t $ becomes a diagonal form $ \vec y^\top D\vec y $ and the RHS of \Cref{eqn:rewritten} becomes
\begin{align}
& \frac{1}{2^L} \int_{[-1,1]^L} \exp\paren{-K^2\lambda\vec t^\top A\vec t} \diff \vec t \notag \\
&= \frac{1}{2^L} \int_{U^{-1}[-1,1]^L} \exp\paren{-K^2\lambda(U\vec y)^\top A (U\vec y)} \cdot |\det(U)| \diff\vec y \label{eqn:app-change-var} \\
&= \frac{1}{2^L} \int_{U^\top[-1,1]^L} \exp\paren{-K^2\lambda\vec y^\top (U^\top A U)\vec y} \diff\vec y \label{eqn:use-ortho} \\
&= \frac{1}{2^L} \int_{U^\top[-1,1]^L} \exp\paren{-K^2\lambda\vec y^\top D\vec y} \diff\vec y \notag \\
&= \frac{1}{2^L} \int_{U^\top[-1,1]^L} \exp\paren{-K^2 \lambda\sum_{i = 1}^{L-1}\vec t(i)^2} \diff\vec t . \label{eqn:how-to-comp} 
\end{align}
\Cref{eqn:app-change-var} is by \Cref{lem:change-var}. 
In \Cref{eqn:use-ortho}, we use the facts that $ U^{-1} = U^\top $ and $ |\det(U)| = 1 $.

\subsection{Laplace's method and proof of~\Cref{lemma:avg_rad_exponent}}
\label{sec:laplace}

To compute \Cref{eqn:how-to-comp}, we note that the integral is degenerate along the direction of the last coordinate $ \vec t(L) $. 
Since the integral domain is bounded, the integral is still finite. 
We first integrate out $ \vec t(L) $ and get an $(L-1)$-dimensional integral w.r.t.\ $ \vec t(1),\cdots,\vec t(L-1) $. 
To this end, observe that for $ \vec t\in U[-1,1]^L $, the last component $ \vec t(L) $ is a function of $ \vec t(1),\cdots,\vec t(L-1) $ and it can take any value of the last coordinate of $ U^\top[-1,1]^L $. 
Therefore the range of $ \vec t(L) $ can be written as $ [g_1(\vec t(1),\cdots,\vec t(L-1)), g_2(\vec t(1),\cdots,\vec t(L-1))] $ where $ g_1(\cdot) $ and $ g_2(\cdot) $ are piecewise linear continuous functions given by $ U $. 
We now integrate out $ \vec t(L) $ and get
\begin{align}
\int_{(U^\top[-1,1]^L)|_{[t_1,\cdots,t_{L-1}]}} e^{-K^2\lambda\sum_{i = 1}^{L-1}t_i^2}(g_2(t_1,\cdots,t_{L-1}) - g_1(t_1,\cdots,t_{L-1})) \diff(t_1,\cdots,t_{L-1}), \label{eqn:to-be-cont-lap} 
\end{align}
where $ (U^\top[-1,1]^L)|_{[t_1,\cdots,t_{L-1}]} \subset\bR^{L-1} $ denotes the set obtained by restricting each vector in $ U^\top[-1,1]^L\subset\bR^L $ to the first $L-1$ coordinates $ (t_1,\cdots,t_{L-1}) $. 

Note that the quadratic function $ f(t_1,\cdots,t_{L-1}) \coloneqq \lambda \sum_{i = 1}^{L-1} t_i^2 $ is nonnegative and attains its unique minimum (which is zero) at $ [t_1,\cdots,t_{L-1}] = [0,\cdots,0] $ which is in the interior of $ U^\top[-1,1]^L $. 
Therefore, by Laplace's method (\Cref{thm:lap}), \Cref{eqn:to-be-cont-lap} converges to 
\begin{align}
\paren{\frac{2\pi}{K^2}}^{\frac{L-1}{2}} \frac{g_2(0,\cdots,0) - g_1(0,\cdots,0)}{\sqrt{\det((\hess f)(0,\cdots,0))}} \notag 
\end{align}
as $ K\to\infty $. 
Since $ \hess f = 2\lambda I_{L-1}\succ 0 $, we have 
\begin{align}
& \paren{\frac{2\pi}{K^2}}^{\frac{L-1}{2}} \frac{g_2(0,\cdots,0) - g_1(0,\cdots,0)}{\sqrt{(2\lambda)^{L-1}}} \notag \\
&= \paren{\frac{\pi}{K^2\lambda}}^{\frac{L-1}{2}} (g_2(0,\cdots,0) - g_1(0,\cdots,0)). \notag  
\end{align}
Note that $ g_2(0,\cdots,0) - g_1(0,\cdots,0) $ is nothing but the length of the range of the last coordinate $ t_{L-1} $ of vectors in $ U^\top[-1,1]^L $. 
Since any vector $ \vec t\in U^\top[-1,1]^L $ can be written as $ U^\top\vec u $ for some $ \vec u\in[-1,1]^L $, the length of the range of the last coordinate of $ \vec t $ is twice the $ \ell_1 $-norm of the last row of $ U^\top $, i.e., the last column of $ U $. 
From \Cref{eqn:matrix-u}, it is not hard to see that 
\begin{align}
g_2(0,\cdots,0) - g_1(0,\cdots,0) &= 
2\cdot L\cdot\frac{1}{\sqrt{L}} 
= 2\sqrt{L}. 
% 2\paren{\frac{1}{\sqrt{L}} + \sum_{j = 1}^{L-1} \frac{1}{\sqrt{j\times(j+1)}}} \eqqcolon 2g_L. 
\label{eqn:gl-def} 
\end{align}
Finally, we get that \Cref{eqn:to-be-cont-lap} (asymptotically) equals
\begin{align}
\paren{\frac{\pi}{K^2\lambda}}^{\frac{L-1}{2}} \cdot2\sqrt{L}. \label{eqn:lapace-result}
\end{align}

Recall that $E(K)$ is the error
 exponent corresponding to \Cref{eqn:avg-rad-bd}. 
Plugging \Cref{eqn:lapace-result} back to \Cref{eqn:how-to-comp} and then back to \Cref{eqn:ldp-to-be-cont}, we have
\begin{align}
E(K) &= \max_{\lambda\ge0}\curbrkt{-\lambda LN - \ln\paren{\frac{1}{2^L}\cdot \paren{\frac{\pi}{K^2\lambda}}^{\frac{L-1}{2}} \cdot2\sqrt{L}} } \notag \\
&= \max_{\lambda\ge0}\curbrkt{-\lambda LN - \ln\paren{\paren{\frac{\pi}{(2K)^2\lambda}}^{\frac{L-1}{2}} \sqrt{L}} } \notag \\
&= \max_{\lambda\ge0}\curbrkt{ -LN\lambda + \frac{L-1}{2}\ln\lambda + (L-1)\ln (2K) \mright. \notag \\
&\qquad\qquad \mleft. - \frac{L-1}{2}\ln\pi - \frac{1}{2}\ln L} \notag \\
&= \max_{\lambda\ge0}\curbrkt{-LN\lambda + \frac{L-1}{2}\ln\lambda} + (L-1)\ln (2K) \notag \\
& \qquad - \frac{L-1}{2}\ln\pi - \frac{1}{2}\ln L \label{eqn:to-be-max} \\
&= -\frac{L-1}{2} + \frac{L-1}{2}\ln\frac{L-1}{2LN} - \frac{L-1}{2}\ln\pi \notag \\
&\qquad + (L-1)\ln (2K) - \frac{1}{2}\ln L \label{eqn:max-lambda} \\
&= \frac{L-1}{2}\ln\frac{L-1}{2\pi eNL} - \frac{1}{2}\ln L + (L-1)\ln (2K). \label{eqn:ldp-result}
\end{align}
\Cref{eqn:max-lambda} follows since the function of $\lambda$ in the maximization in \Cref{eqn:to-be-max} is convex and attains its maximum at $ \lambda = \frac{L-1}{2LN} $. This completes the proof.
\qed

\subsection{Finite codebook with minimum average squared radius $nN$ and proof of~\Cref{lemma:finitecodebook_avgrad}}
\label{sec:together}

%Since, by the definition of PPPs (\Cref{def:ppp}), the number $M$ of points in $ \cA $ is Poissonly distributed with mean $ \lambda_n|\cA| = \lambda_n(2K)^n $, the expected number of lists with average squared radius at most $ nN $ is 
Since the number $M$ of points in $ \cA $ is $ \lambda_n|\cA| = \lambda_n(2K)^n $, the expected number of lists with average squared radius at most $ nN $ is 
\begin{align}
& \exptover{}{\card{\curbrkt{\cL\in\binom{\cC_K'}{L}:\ol\rad^2(\cL)\le nN}}} \notag \\
&= 
\exptover{\vbfx_1,\cdots,\vbfx_M}{\card{\curbrkt{\cL\in\binom{[M]}{L}:\ol\rad^2(\curbrkt{\vbfx_i}_{i\in\cL})\le nN}} 
} \label{eqn:cond-expt} \\
  &=  \binom{M}{L}e^{-nE(K)}  \notag \\
  &= \frac{M^Le^{o(n)}}{L!}e^{-nE(K)} \notag\\
&= \frac{\lambda_n^L(2K)^{nL}e^{-nE(K)+o(n)}}{L!}. \label{eqn:expt-bad}
\end{align}
%\Cref{eqn:cond-expt} follows from law of total expectation where the outer expectation is over $ M\sim\pois(\lambda_n|\cA|) $ and the inner expectation is over $ \vbfx_i\iid\unif(\cA) $ for $ 1\le i\le M $. 
%In \Cref{eqn:change-of-var}, we change variable $ \ell = m - L $. 
%In \Cref{eqn:total-prob}, the summation equals one since it is the total probability of a Poisson random variable with mean $ \lambda_n|\cA| $. 

We now set $ \lambda_n $ in such a way that \Cref{eqn:expt-bad} is at most $ \frac{1}{2}\expt{M} = \frac{1}{2}\lambda_n|\cA| = \frac{1}{2}\lambda_n(2K)^n $. 
That is, 
\begin{align}
&& \frac{1}{2} \lambda_n(2K)^n &\ge \frac{\lambda_n^L(2K)^{nL}e^{-nE(K)+o(n)}}{L!} \notag \\
\impliedby&& L!/2 &\ge \lambda_n^{L-1} (2K)^{(L-1)n} e^{-nE(K)+o(n)} \notag \\
\impliedby&& \lambda_n &\le (L!/2)^{\frac{1}{L-1}}\frac{e^{nE(K)/(L-1)+o(n)}}{(2K)^n} \notag \\
&& &= (L!/2)^{\frac{1}{L-1}} e^{n\paren{\frac{E(K)}{L-1} - \ln(2K)}+o(n)} \notag \\
&& &= (L!/2)^{\frac{1}{L-1}} \exp\paren{n\paren{ \frac{1}{2}\ln\frac{L-1}{2\pi eNL} -\frac{1}{2(L-1)}\ln L }+o(n)}. \label{eqn:lambda-ppp}
\end{align}
% or 
% \begin{align}
% \lambda_n \le (L!/2)^{\frac{1}{L-1}}\frac{e^{nE(K)/(L-1)}}{(2K)^n} = (L!/2)^{\frac{1}{L-1}} e^{n\paren{\frac{E(K)}{L-1} - \ln(2K)}} . 
% \label{eqn:lambda-ppp}
% \end{align}
After expurgating out one codeword from each bad list, we get an $ (N,L-1) $-average-radius multiple packing $\cC_K$ of size at least $ \frac{1}{2}\expt{M} $ and the density $\frac{1}{n}\ln\frac{|\cC_K|}{|\cA|}$ is therefore at least
\begin{align}
\frac{1}{n}\ln\frac{\frac{1}{2}\expt{M}}{|\cA|} &= \frac{1}{n}\ln\frac{\frac{1}{2}\lambda_n|\cA|}{|\cA|} = \frac{1}{n}\ln\frac{\lambda_n}{2} . 
% \asymp \frac{E(K)}{L - 1} - \ln(2K) . 
\label{eqn:put-together} 
\end{align}
Substituting \Cref{eqn:lambda-ppp} here, we get the following lower bound on the density
\begin{align}
% \frac{1}{L-1}\paren{\frac{L-1}{2}\ln\frac{L-1}{2\pi eNL} - \frac{1}{2}\ln L + (L-1)\ln (2K)} - \ln (2K) = 
\frac{1}{2}\ln\frac{L-1}{2\pi eNL} - \frac{1}{2(L-1)}\ln L +o(1), \label{eqn:desired-density}
\end{align}
as promised in \Cref{lemma:finitecodebook_avgrad}. 

\subsection{Unbounded packing and proof of~\Cref{thm:lb-ppp}}
\label{sec:unbdd-packing}

% !TEX root = ./multipack-ppp-journal.tex

The above derivation shows the existence of a \emph{finite} codebook $ \cC\cap[-K,K]^n $ in which all $L$-tuple of points have radius at least $\sqrt{nN}$. 
To obtain an \emph{unbounded} $(N,L-1)$-packing, let us take $ K=n^2 $ and define $\cC_n \coloneqq \cC\cap[-n^2,n^2]^n$ and 
\begin{align}
\cC_\infty &\coloneqq \cC_n + (n^2 + n^{0.6}) \bZ^n . \notag 
\end{align}
In words, $ \cC_\infty $ is obtained by tiling $ \bR^n $ using translations of $ \cC_n $ and leaving a gap of width $ 2\cdot n^{0.6} $ between adjacent copies of $ \cC_n $. 
The NLD of $ \cC_\infty $ is essentially the same as that of $ \cC\cap[-n^2,n^2]^n $ which is given by \Cref{eqn:desired-density}. 
Indeed, since $ \cC_\infty $ is periodic, we have
\begin{align}
	R(\cC_\infty) &= \frac{1}{n} \ln \frac{|\cC_n|}{\abs{[-(n^2+n^{0.6}),(n^2+n^{0.6})]^n}} 
	= \frac{1}{n} \ln \frac{|\cC_n|}{\abs{[-n^2,n^2]^n}} + \frac{1}{n}\ln\frac{(2n^2)^n}{(2(n^2+n^{0.6}))^n} 
	\xrightarrow{n\to\infty} R(\cC_n) . \notag 
\end{align}
Moreover, we claim that $ \cC_\infty $ is an $(N,L-1)$-packing. 
To see this, take any $ \cL\subset\binom{\cC_\infty}{L} $. 
If $ \cL\subset\cC_n + \vz $ for some $ \vz\in(n^2+n^{0.6})\bZ^n $, then $ \rad(\cL)\ge\sqrt{nN} $ by the guarantee of $ \cC_n $. 
Otherwise, there exist two points $ \vx_1,\vx_2\in\cL $ such that $ \vx_1\in\cC_n + \vz_1 $ and $ \vx_2\in\cC_n + \vz_2 $ for two distinct $ \vz_1\ne\vz_2\in(n^2+n^{0.6})\bZ_n $. 
Then 
\begin{align}
\rad(\cL) &\ge \frac{1}{2}\normtwo{\vx_1 - \vx_2} \ge n^{0.6} \ge \sqrt{nN}. \notag 
\end{align}
Therefore, we obtain an $(N,L-1)$-packing $ \cC_\infty $ of NLD asymptotically equal to \Cref{eqn:desired-density}. 
The proof of \Cref{thm:lb-ppp} is complete. 

%\begin{remark}
%\label{rk:tiling}
%The derivations in \Cref{sec:ldp,sec:laplace,sec:together} show that for every sufficiently large $K>0$, there exists an (infinite) PPP $\cC$ such that every $L$-tuple of points in $ \cC\cap[-K,K]^n $ has radius at least $ \sqrt{nN} $. 
%However, we do not manage to find a rigorous way to pass to the limit as $K\to\infty$ and argue that $\cC$ itself as an \emph{infinite} point set is an $(N,L-1)$-multiple packing. 
%Indeed, it is claimed in \cite[Note 5]{shlosman-tsfasman-2000-random-packing} that $ \cC\cap[-K,K]^n $ may not converge as $K\to\infty$. 
%For this reason, in \Cref{sec:unbdd-packing} we pass to an infinite packing by tiling the whole space using translations of $ \cC\cap[-K,K]^n $ while leaving a sufficiently large gap (specifically $ 2\cdot n^{0.6} $). 
%This avoids the interaction of points from different translations and guarantees $(N,L-1)$-list-decodability of the infinite point set. 
%As opposed to an infinite PPP, the tiling forms a periodic packing, therefore we do not encounter limiting issues. 
%The density remains essentially the same by taking $K$ to be sufficiently large (specifically $ n^2 $). 
%\end{remark}

\subsection{{Alternate approaches for bounding \Cref{eqn:ppp-to-bound-clean}  }}
We managed to compute the exact asymptotics (up to lower order terms in the exponent) of the tail probability given by \Cref{eqn:ppp-to-bound-clean}. 
The way we did so is by applying the large deviation principle and performing ad hoc calculations on the moment generating function of the random quadratic form of interest. 
From the perspective of concentration of measure, the tail probability we computed can be cast from several different angles: 
\begin{enumerate}
	\item Gaussian integral w.r.t.\ a general (non-identity) degenerate covariance matrix $A$;
	\item Concentration of the uniform measure on a solid cube (a useful trick for which is to push it forward to the Gaussian measure and apply Lipschitz concentration \cite{bobkov-2010-concentration-on-cube});
	\item The (standard) Gaussian measure of a parallelepiped defined by the linear transformation $U$; 
	\item The probability that a Gaussian (with zero mean and general covariance matrix) lies in a cube;
	\item Hanson--Wright inequality for quadratic forms in subgaussian random vectors \cite{vershynin-2018-high-dim-prob-book} which, in our case, are uniform vectors in a solid cube. 
\end{enumerate}
We tried all the above techniques. 
However, they do not seem to yield the correct exponent, at least in their vanilla forms, though they may give certain exponentially decaying bounds. 
Therefore, we feel that \Cref{eqn:ppp-to-bound-clean} is a cute example for which standard concentration tools are not able to produce the optimal bound.

\subsection{Connections to \cite{blinovsky-2005-random-packing}}
\label{sec:bli-mistake-ppp}
The paper \cite{blinovsky-2005-random-packing} analyzed the list-decodability of expurgated PPPs and  arrived at the same bound (\Cref{eqn:lb-ppp}) as ours, and in fact the current paper was inspired by~\cite{blinovsky-2005-random-packing}.

However, there were some gaps in the proof of~\cite{blinovsky-2005-random-packing} that we were not able to resolve.
In the paper, it was shown that for every sufficiently large $K>0$, there exists an (infinite) codebook $\cC$ obtained by expurgating a PPP such that every $L$-tuple of points in $ \cC\cap[-K,K]^n $ has radius at least $ \sqrt{nN} $. 
However, we could not find a rigorous way to pass to the limit as $K\to\infty$ and argue that $\cC$ itself as an \emph{infinite} point set is an $(N,L-1)$-multiple packing. 
Indeed, it is claimed in \cite[Note 5]{shlosman-tsfasman-2000-random-packing} that $ \cC\cap[-K,K]^n $ may not converge as $K\to\infty$. 
%For this reason, in \Cref{sec:unbdd-packing} we pass to an infinite packing by tiling the whole space using translations of $ \cC\cap[-K,K]^n $ while leaving a sufficiently large gap (specifically $ 2\cdot n^{0.6} $). 
%This avoids the interaction of points from different translations and guarantees $(N,L-1)$-list-decodability of the infinite point set. 
%As opposed to an infinite PPP, the tiling forms a periodic packing, therefore we do not encounter limiting issues. 
%The density remains essentially the same by taking $K$ to be sufficiently large (specifically $ n^2 $). 
% Our approach bears few similarities with his besides the fact that we both analyzed the tail probability of the average squared radius (\Cref{eqn:avg-rad-bd}). 
% In fact, Blinovsky's approach suffers fundamental flaws which invalidate his calculations. 
% We explain his mistakes below. 

Although our proof also involves analyzing the tail probability of the average squared radius (\Cref{eqn:avg-rad-bd}), our techniques are different, as outlined below. 
%From our understanding of \cite{blinovsky-2005-random-packing}, we believe that there are some flaws which are outlined below.

Let $ \cC $ be a PPP (without expurgation yet) and $ K\in\bR $. 
Let $ \vbfx_1,\cdots,\vbfx_L\in\cC\cap[-K,K]^n $ be an $L$-list.
Recall that they are independent and uniformly distributed in $ [-K,K]^n $. 
Let $ \vbfx = \frac{1}{L}\sum_{i=1}^L\vbfx_i $ denote the centroid of the list. 
To compute \Cref{eqn:avg-rad-bd}, \cite{blinovsky-2005-random-packing} claimed that we could use an orthogonal transformation $ \vbfx_i\mapsto\vbfu_i $ ($ 1\le i\le L $) to the list so that $ \frac{1}{L}\normtwo{\vbfu_1} = \normtwo{\vbfu} $ where $ \vbfu \coloneqq \frac{1}{L}\sum_{i = 1}^L\normtwo{\vbfu_i} $. This is in contrast to our approach. 
However, it should be noted that an orthogonal transformation only reflects and/or rotates the list, but does not translate it.

From our understanding of the paper \cite{blinovsky-2005-random-packing}, the ideas can be interpreted as follows.
% Interpreted more generously, the idea in \cite{blinovsky-2005-random-packing} can be described as follows. 
Since the average squared radius is invariant under rigid transformations (i.e., translations, rotations, reflections and their combination) and a homogeneous PPP is stationary and isotropic, \cite{blinovsky-2005-random-packing} attempts to transform the list rigidly so that the resulting average squared radius admits a simpler expression and the list is still independent and uniformly distributed in the cube. However, it appears that such a rigid transformation does not exist. We instead use a different, much simpler approach by first constructing a finite codebook and then tiling this. The high-level construction is similar to~\cite{poltyrev1994coding} which was originally studied for the problem of reliable communication over additive-white Gaussian noise channels.

\section{Open questions}
\label{sec:open}

We end the paper with several intriguing open questions.
\begin{enumerate}
	\item \label{itm:lp-packing}
	The problem of packing spheres in $ \ell_p $ space was also addressed in the literature \cite{rankin-sphericalcap-1955,spence-1970-lp-packing,ball-1987-lp-packing,samorodnitsky-l1}. 
	Recently, there was an exponential improvement on the optimal packing density in $ \ell_p $ space \cite{sah-2020-lp-sphere-packing} relying on the Kabatiansky--Levenshtein bound \cite{kabatiansky-1978}. 
	It is worth exploring the $ \ell_p $ version of the multiple packing problem. 
	One obstacle here is that the $ \ell_p $ average radius does not admit a closed form expression unlike the $p=2$ case.
	
	% \item LP bound for Lee codes. Polyanskiy's trick for Lee codes. 
	
	% \item \label{itm:khinchin-beta}
	% We examined the performance of multiple packings obtained from the Gaussian distribution (\Cref{sec:lb-gaussian}), the uniform distribution on a sphere (\Cref{sec:lb-spherical}) and the uniform distribution in a ball (\Cref{sec:ball-codes}), respectively. 
	% It turns out that these distributions are special cases of the $n$-dimensional beta distribution $ \Beta_n(\beta) $ when $ \beta = \infty,-1,0 $, respectively. 
	% If one can prove a Khinchin's inequality for beta distribution with sharp constants, then it allows us to unify our proofs for three distributions and potentially strengthen our bounds for spherical codes and ball codes. 
	% Also, Khinchin's inequality for beta distribution is of independent interest for probabilists. 
	
	% \item Lattice error exponents \cite{erez-zamir-2004,liu-2006-lattice-ee}. 
	
	% \item List-decoding error exponents of typical random codes. 
	
	\item \label{itm:ld-cap-avgrad} 
	In this paper, we treat (regular) list-decoding and average-radius list-decoding as two different notions and obtain bounds for the latter (which automatically lower bounds the former). 
	It follows from our bounds that 
	% in the \emph{unbounded} case, 
	the largest multiple packing density under these two notions coincide as $ L\to\infty $. 
	% However, as far as we know, it is unknown whether in the \emph{bounded} case, the largest multiple packing density under standard and average-radius list-decoding is the same in the large $L$ limit. 
	% Indeed, we do not have a pair of lower and upper bounds for both notions that converge to the same limiting value. 
	However, as far as we know, it is unknown whether the largest multiple packing density under standard and average-radius list-decoding is the same for any \emph{finite} $L$.
	 % even in the unbounded case and of course in the bounded case. 

	% \item \label{itm:cheb-rad}
	% For the regular notion of list-decoding, our best lower bound is proved via a peculiar connection with error exponents. 
	% We do not know how to directly analyze the tail probability of the Chebyshev radius, even for Gaussian codes. 
	% One can view it as the tail of the maximum of a certain Gaussian process. 
	% This looks like a proper venue where the chaining method \cite{vanhandel-2014-high-dim-prob} is applicable. 
	% However, it seems unlikely that one can extract a meaningful exponent using the generic chaining machinery. 
	% Note that for the purpose of maximizing the rate, we do care about the exact exponent, not only an exponentially decaying bound. 

\end{enumerate}

% Knowns/Results/TODOs:
% \begin{enumerate}
% 	% \item GV via (expurgated) spherical codes;
% 	% \item GV via (expurgated) Gaussian codes;
% 	% \item GV via lattice codes;
% 	% \item GV via volume argument/greedy algorithm;
% 	\item Constrained
% 	\begin{enumerate}
% 		\item Lower bounds
% 		\begin{enumerate}
			
% 			\item Gaussian (average-radius)
% 			\item Spherical (average-radius)
% 			\item Ball (average-radius)
% 			\item Standard list-decoding: error exponent (Gallager's EE via spherical codes, lattice)
% 		\end{enumerate}
% 		\item Upper bounds
% 		\begin{enumerate}
% 			\item Plotkin + EB (Blinovsky and my correction, check Blachman--Few's UB)
% 			\item Distance distribution (Blinovsky--Litsyn)
% 		\end{enumerate}
% 	\end{enumerate}
% 	\item Unconstrained
% 	\begin{enumerate}
% 		\item Lower bounds
% 		\begin{enumerate}
% 			\item PPP (average-radius)
% 			\item Other point processes (average-radius)?
% 			\item Standard list-decoding: error exponent (Poltyrev's IC, lattice, PPP and other point processes)
% 			\item Blachman--Few's LB via dilation
% 		\end{enumerate}
% 		\item Upper bounds
% 	\end{enumerate}
% 	\item Binary
% 	\begin{enumerate}
% 		\item Lower bounds
% 		\begin{enumerate}
% 			\item Average-radius (Blinovsky)
% 			\item Standard list-decoding (ZBJ)
% 			\item Error exponent??
% 		\end{enumerate}
% 		\item Upper bounds
% 		\begin{enumerate}
% 			\item Plotkin + EB (Blinovsky, ZBJ)
% 			\item Distance distribution, Polyanskiy
% 		\end{enumerate}
% 	\end{enumerate}
% 	\item Concatenated codes and Zyablov bound;
% \end{enumerate}
% Relation between densities of constrained/unconstrained packing

\section{Acknowledgement}
% YZ thanks Tomasz Tkocz for discussions on Khinchin's inequality with sharp constants for sums and quadratic forms. 
YZ thanks Jiajin Li for making the observation given by \Cref{eqn:cheb-rad-interchange}. 
He also would like to thank Nir Ailon and Ely Porat for several helpful conversations throughout this project, and Alexander Barg for insightful comments on the manuscript. 

YZ has received funding from the European Union's Horizon 2020 research and innovation programme under grant agreement No 682203-ERC-[Inf-Speed-Tradeoff]. 
The work of SV was supported by a seed grant from IIT Hyderabad and the start-up research grant from the Science and Engineering Research Board, India (SRG/2020/000910). 

% \emph{YZ dedicates this work to the memory of the grandmother of Jasmine Yao. }

\bibliographystyle{alpha}
\bibliography{ref} 

\newcommand{\etalchar}[1]{$^{#1}$}
\begin{thebibliography}{BADTS20}

\bibitem[AB08]{ahlswede-blinovsky-2008-book}
Rudolf Ahlswede and Vladimir Blinovsky.
\newblock {\em Lectures on advances in combinatorics}.
\newblock Universitext. Springer-Verlag, Berlin, 2008.

\bibitem[ABL00]{abl-2000-list-size-2}
Alexei Ashikhmin, Alexander Barg, and Simon Litsyn.
\newblock A new upper bound on codes decodable into size-2 lists.
\newblock In {\em Numbers, Information and Complexity}, pages 239--244.
  Springer, 2000.

\bibitem[BADTS20]{ben-aroya-doron-ta-shma-2018-explicit-erasure-ld}
Avraham Ben-Aroya, Dean Doron, and Amnon Ta-Shma.
\newblock Near-optimal erasure list-decodable codes.
\newblock In {\em 35th Computational Complexity Conference (CCC 2020)}. Schloss
  Dagstuhl-Leibniz-Zentrum f{\"u}r Informatik, 2020.

\bibitem[Bal87]{ball-1987-lp-packing}
Keith Ball.
\newblock Inequalities and sphere-packing inl p.
\newblock {\em Israel Journal of Mathematics}, 58(2):243--256, 1987.

\bibitem[BBJ19]{bhattacharya2019shared}
Sagnik Bhattacharya, Amitalok~J Budkuley, and Sidharth Jaggi.
\newblock Shared randomness in arbitrarily varying channels.
\newblock In {\em 2019 IEEE International Symposium on Information Theory
  (ISIT)}, pages 627--631. IEEE, 2019.

\bibitem[BF63]{blachman-few-1963-multiple-packing}
NM~Blachman and L~Few.
\newblock Multiple packing of spherical caps.
\newblock {\em Mathematika}, 10(1):84--88, 1963.

\bibitem[Bli86]{blinovsky-1986-ls-lb-binary}
Vladimir~M Blinovsky.
\newblock {Bounds for codes in the case of list decoding of finite volume}.
\newblock {\em {Problems of Information Transmission}}, 22:7--19, 1986.

\bibitem[Bli99]{blinovsky-1999-list-dec-real}
V~Blinovsky.
\newblock Multiple packing of the euclidean sphere.
\newblock {\em IEEE Transactions on Information Theory}, 45(4):1334--1337,
  1999.

\bibitem[Bli05a]{blinovsky-2005-ls-lb-qary}
Vladimir~M Blinovsky.
\newblock {Code bounds for multiple packings over a nonbinary finite alphabet}.
\newblock {\em {Problems of Information Transmission}}, 41:23--32, 2005.

\bibitem[Bli05b]{blinovsky-2005-random-packing}
Vladimir~M Blinovsky.
\newblock Random sphere packing.
\newblock {\em Problems of Information Transmission}, 41(4):319--330, 2005.

\bibitem[Bli08]{blinovsky-2008-ls-lb-qary-supplementary}
Vladimir~M Blinovsky.
\newblock {On the convexity of one coding-theory function}.
\newblock {\em {Problems of Information Transmission}}, 44:34--39, 2008.

\bibitem[Bli12]{blinovsky2012book}
Volodia Blinovsky.
\newblock {\em Asymptotic combinatorial coding theory}, volume 415.
\newblock Springer Science \& Business Media, 2012.

\bibitem[Bob10]{bobkov-2010-concentration-on-cube}
Sergey~G Bobkov.
\newblock On concentration of measure on the cube.
\newblock {\em Journal of Mathematical Sciences}, 165(1):60--70, 2010.

\bibitem[CKM{\etalchar{+}}17]{cohn-2017-24spherepacking}
Henry Cohn, Abhinav Kumar, Stephen~D Miller, Danylo Radchenko, and Maryna
  Viazovska.
\newblock The sphere packing problem in dimension 24.
\newblock {\em Annals of Mathematics}, pages 1017--1033, 2017.

\bibitem[CS13]{conway-sloane-book}
John~Horton Conway and Neil James~Alexander Sloane.
\newblock {\em {Sphere packings, lattices and groups}}, volume 290.
\newblock Springer Science \& Business Media, 2013.

\bibitem[Del73]{delsarte-1973}
Philippe Delsarte.
\newblock An algebraic approach to the association schemes of coding theory.
\newblock {\em Philips Res. Rep. Suppl.}, 10:vi+--97, 1973.

\bibitem[Eli57]{elias-1957-listdec}
Peter Elias.
\newblock {\em List decoding for noisy channels}.
\newblock Massachusetts Institute of Technology, Research Laboratory of
  Electronics, Cambridge, Mass., 1957.
\newblock Rep. No. 335.

\bibitem[GHS20]{guruswami-2020-listdec-insdel}
Venkatesan Guruswami, Bernhard Haeupler, and Amirbehshad Shahrasbi.
\newblock Optimally resilient codes for list-decoding from insertions and
  deletions.
\newblock In {\em Proceedings of the 52nd Annual ACM SIGACT Symposium on Theory
  of Computing}, pages 524--537, 2020.

\bibitem[Gil52]{gilbert1952}
Edgar~N Gilbert.
\newblock A comparison of signalling alphabets.
\newblock {\em The Bell system technical journal}, 31(3):504--522, 1952.

\bibitem[Gop77]{goppa1977}
Valerii~Denisovich Goppa.
\newblock Codes associated with divisors.
\newblock {\em Problemy Peredachi Informatsii}, 13(1):33--39, 1977.

\bibitem[GP12]{grigorescu-peikert-2012-list-dec-barnes-wall}
Elena Grigorescu and Chris Peikert.
\newblock List decoding barnes-wall lattices.
\newblock In {\em 2012 IEEE 27th Conference on Computational Complexity}, pages
  316--325. IEEE, 2012.

\bibitem[Gur06]{guruswami-it2003}
V~Guruswami.
\newblock List decoding from erasures: bounds and code constructions.
\newblock {\em IEEE Transactions on Information Theory}, 49(11):2826--2833,
  2006.

\bibitem[HAB{\etalchar{+}}17]{hales2017formal}
Thomas Hales, Mark Adams, Gertrud Bauer, Tat~Dat Dang, John Harrison, Hoang
  Le~Truong, Cezary Kaliszyk, Victor Magron, Sean McLaughlin, Tat~Thang Nguyen,
  et~al.
\newblock A formal proof of the kepler conjecture.
\newblock In {\em Forum of mathematics, Pi}, volume~5. Cambridge University
  Press, 2017.

\bibitem[HF11]{hales1998kepler}
Thomas Hales and Samuel Ferguson.
\newblock {\em The {K}epler conjecture}.
\newblock Springer, New York, 2011.
\newblock The Hales-Ferguson proof, Including papers reprinted from Discrete
  Comput. Geom. {{\bf{3}}6} (2006), no. 1, Edited by Jeffrey C. Lagarias.

\bibitem[HK19]{hosseinigoki-kosut-2018-oblivious-gaussian-avc-ld}
Fatemeh Hosseinigoki and Oliver Kosut.
\newblock List-decoding capacity of the {G}aussian arbitrarily-varying channel.
\newblock {\em Entropy}, 21(6):Paper No. 575, 16, 2019.

\bibitem[Hug97]{hughes-1997-list-avc}
Brian~L. Hughes.
\newblock The smallest list for the arbitrarily varying channel.
\newblock {\em IEEE Transactions on Information Theory}, 43(3):803--815, 1997.

\bibitem[Jos58]{joshi1958singleton}
D.~D. Joshi.
\newblock A note on upper bounds for minimum distance codes.
\newblock {\em Information and Control}, 1:289--295, 1958.

\bibitem[Kep11]{kepler-1611}
Johannes Kepler.
\newblock Strena seu de nive sexangula (the six-cornered snowflake).
\newblock {\em Frankfurt: Gottfried. Tampach}, 1611.

\bibitem[KL78]{kabatiansky-1978}
Grigorii~Anatolevich Kabatiansky and Vladimir~Iosifovich Levenshtein.
\newblock {On bounds for packings on a sphere and in space}.
\newblock {\em Problemy Peredachi Informatsii}, 14(1):3--25, 1978.

\bibitem[KL95]{kalai-linial-1995-distance-distribution}
Gil Kalai and Nathan Linial.
\newblock On the distance distribution of codes.
\newblock {\em IEEE Transactions on Information Theory}, 41(5):1467--1472,
  1995.

\bibitem[Kom53]{komamiya1953singleton}
Y~Komamiya.
\newblock Application of logical mathematics to information theory.
\newblock {\em Proc. 3rd Japan. Nat. Cong. Appl. Math}, 437, 1953.

\bibitem[Lan04]{langberg-focs2004}
M.~Langberg.
\newblock Private codes or succinct random codes that are (almost) perfect.
\newblock In {\em 45th Annual IEEE Symposium on Foundations of Computer
  Science}, pages 325--334, 2004.

\bibitem[Lit99]{litsyn-1999}
Simon Litsyn.
\newblock New upper bounds on error exponents.
\newblock {\em IEEE Transactions on Information Theory}, 45(2):385--398, 1999.

\bibitem[Min10]{minkowski-sphere-pack}
Hermann Minkowski.
\newblock {\em Geometrie der zahlen}.
\newblock BG Teubner, 1910.

\bibitem[MP22]{mook-peikert-2020-lattice}
Ethan Mook and Chris Peikert.
\newblock Lattice (list) decoding near {M}inkowski's inequality.
\newblock {\em IEEE Trans. Inform. Theory}, 68(2):863--870, 2022.

\bibitem[MRRW77]{mrrw2}
Robert McEliece, Eugene Rodemich, Howard Rumsey, and Lloyd Welch.
\newblock New upper bounds on the rate of a code via the delsarte-macwilliams
  inequalities.
\newblock {\em IEEE Transactions on Information Theory}, 23(2):157--166, 1977.

\bibitem[Pol94]{poltyrev1994coding}
Gregory Poltyrev.
\newblock On coding without restrictions for the awgn channel.
\newblock {\em IEEE Transactions on Information Theory}, 40(2):409--417, 1994.

\bibitem[Pol16]{polyanskiy-2016-list-dec}
Yury Polyanskiy.
\newblock Upper bound on list-decoding radius of binary codes.
\newblock {\em IEEE Transactions on Information Theory}, 62(3):1119--1128,
  2016.

\bibitem[PZ21]{polyanskii-zhang-2021-z}
Nikita Polyanskii and Yihan Zhang.
\newblock Codes for the z-channel.
\newblock {\em arXiv preprint arXiv:2105.01427}, 2021.

\bibitem[Ran55]{rankin-sphericalcap-1955}
R.A. Rankin.
\newblock {The closest packing of spherical caps in n dimensions}.
\newblock {\em Proc.\ Glasgow Math.\ Assoc.}, 2:139--144, 1955.

\bibitem[RS60]{reed-solomon}
Irving~S Reed and Gustave Solomon.
\newblock Polynomial codes over certain finite fields.
\newblock {\em Journal of the society for industrial and applied mathematics},
  8(2):300--304, 1960.

\bibitem[RYZ22]{resch-yuan-zhang-plotkin}
Nicolas Resch, Chen Yuan, and Yihan Zhang.
\newblock Zero-rate thresholds and new capacity bounds for list-decoding and
  list-recovery.
\newblock {\em arXiv preprint arXiv:2210.07754}, 2022.

\bibitem[Sam13]{samorodnitsky-l1}
Alex Samorodnitsky.
\newblock A bound on l1 codes.
\newblock \url{https://www.cs.huji.ac.il/~salex/papers/L1_codes.pdf}, 2013.

\bibitem[Sar08]{sarwate-thesis}
Anand~D. Sarwate.
\newblock {\em Robust and adaptive communication under uncertain interference}.
\newblock PhD thesis, EECS Department, University of California, Berkeley, Jul
  2008.

\bibitem[SG12]{sarwate-gastpar-2012-listdec}
Anand~D Sarwate and Michael Gastpar.
\newblock List-decoding for the arbitrarily varying channel under state
  constraints.
\newblock {\em IEEE transactions on information theory}, 58(3):1372--1384,
  2012.

\bibitem[Sin64]{singleton1964}
Richard Singleton.
\newblock Maximum distance q-nary codes.
\newblock {\em IEEE Transactions on Information Theory}, 10(2):116--118, 1964.

\bibitem[Spe70]{spence-1970-lp-packing}
E~Spence.
\newblock Packing of spheres in lp.
\newblock {\em Glasgow Mathematical Journal}, 11(1):72--80, 1970.

\bibitem[SSSZ20]{sah-2020-lp-sphere-packing}
Ashwin Sah, Mehtaab Sawhney, David Stoner, and Yufei Zhao.
\newblock Exponential improvements for superball packing upper bounds.
\newblock {\em Advances in Mathematics}, 365:107056, 2020.

\bibitem[ST01]{shlosman-tsfasman-2000-random-packing}
Senya Shlosman and Michael~A. Tsfasman.
\newblock Random lattices and random sphere packings: typical properties.
\newblock {\em Mosc. Math. J.}, 1(1):73--89, 2001.

\bibitem[Thu11]{thue1911-2dspherepacking}
Axel Thue.
\newblock {\em {\ "U} about the densest compilation of congruent circles in a
  plane}.
\newblock Number~1. J. Dybwad, 1911.

\bibitem[T{\'o}t40]{toth-1940-2dspherepacking}
L~Fejes T{\'o}th.
\newblock Uber einen geometrischen satz.
\newblock {\em Math}, 2(46):79--83, 1940.

\bibitem[TVZ82]{tvz}
Michael~A Tsfasman, SG~Vl{\u{a}}dutx, and Th~Zink.
\newblock Modular curves, shimura curves, and goppa codes, better than
  varshamov-gilbert bound.
\newblock {\em Mathematische Nachrichten}, 109(1):21--28, 1982.

\bibitem[Var57]{varshamov1957}
Rom~Rubenovich Varshamov.
\newblock Estimate of the number of signals in error correcting codes.
\newblock {\em Docklady Akad. Nauk, SSSR}, 117:739--741, 1957.

\bibitem[Ver18]{vershynin-2018-high-dim-prob-book}
Roman Vershynin.
\newblock {\em High-dimensional probability: An introduction with applications
  in data science}, volume~47.
\newblock Cambridge university press, 2018.

\bibitem[Via17]{viazovska-2017-8dspherepacking}
Maryna~S Viazovska.
\newblock The sphere packing problem in dimension 8.
\newblock {\em Annals of Mathematics}, pages 991--1015, 2017.

\bibitem[Woz58]{wozencraft-1958-listdec}
John~M Wozencraft.
\newblock List decoding.
\newblock {\em Quarterly Progress Report}, 48:90--95, 1958.

\bibitem[ZBJ20]{zhang-2019-list-dec-general}
Yihan Zhang, Amitalok~J. Budkuley, and Sidharth Jaggi.
\newblock {Generalized List Decoding}.
\newblock In Thomas Vidick, editor, {\em 11th Innovations in Theoretical
  Computer Science Conference (ITCS 2020)}, volume 151 of {\em Leibniz
  International Proceedings in Informatics (LIPIcs)}, pages 51:1--51:83,
  Dagstuhl, Germany, 2020. Schloss Dagstuhl--Leibniz-Zentrum fuer Informatik.

\bibitem[ZJB20]{zhang-2020-obli-list-dec}
Yihan Zhang, Sidharth Jaggi, and Amitalok~J Budkuley.
\newblock {Tight List-Sizes for Oblivious AVCs under Constraints}.
\newblock {\em arXiv preprint arXiv:2009.03788}, 2020.

\bibitem[ZV22a]{zhang-ppp-isit}
Yihan Zhang and Shashank Vatedka.
\newblock List-decodability of poisson point processes.
\newblock In {\em 2022 IEEE International Symposium on Information Theory
  (ISIT)}, pages 2559--2564, 2022.

\bibitem[ZV22b]{zhang-vatedka-2019-listdecreal}
Yihan Zhang and Shashank Vatedka.
\newblock List decoding random euclidean codes and infinite constellations.
\newblock {\em IEEE Transactions on Information Theory}, pages 1--1, 2022.

\bibitem[ZV22c]{zhang-split-misc}
Yihan Zhang and Shashank Vatedka.
\newblock Multiple packing: Lower and upper bounds.
\newblock {\em arXiv preprint arXiv:2211.04406}, 2022.

\bibitem[ZV22d]{zhang-split-ee}
Yihan Zhang and Shashank Vatedka.
\newblock Multiple packing: Lower bounds via error exponents.
\newblock {\em arXiv preprint arXiv:2211.04408}, 2022.

\bibitem[ZVJ20]{zhang-2020-twoway}
Yihan Zhang, Shashank Vatedka, and Sidharth Jaggi.
\newblock Quadratically constrained two-way adversarial channels.
\newblock {\em arXiv preprint arXiv:2001.02575}, 2020.

\bibitem[ZVJS22]{zhang-quadratic-arxiv}
Yihan Zhang, Shashank Vatedka, Sidharth Jaggi, and Anand~D. Sarwate.
\newblock Quadratically constrained myopic adversarial channels.
\newblock {\em IEEE Transactions on Information Theory}, 68(8):4901--4948,
  2022.

\end{thebibliography}

\end{document}